\documentclass[12pt,reqno]{amsart}%
\usepackage[a4paper, margin=2.9cm]{geometry}
\usepackage{amsmath}
\usepackage{amsfonts}
\usepackage{amssymb}
\usepackage{graphicx}
\usepackage{xcolor}
\usepackage{multirow}%
\providecommand{\U}[1]{\protect\rule{.1in}{.1in}}
\newtheorem{theorem}{Theorem}
\newtheorem*{theorem*}{Theorem}

\newtheorem{corollary}[theorem]{Corollary}

\newtheorem{definition}[theorem]{Definition}

\newtheorem{lemma}[theorem]{Lemma}

\newtheorem{remark}[theorem]{Remark}

\allowdisplaybreaks
\begin{document}
\title[Euler--MacLaurin summation formula and Bernoulli polynomials]{Euler--MacLaurin summation formula on polytopes and expansions in multivariate
Bernoulli polynomials}
\author[L. Brandolini]{L. Brandolini$^1$}
\address{\phantom{i}$^1$Dipartimento di Ingegneria Gestionale, dell'Informazione e
della Produzione, Universit{\`a} degli Studi di Bergamo, Viale G. Marconi 5, 24044,
Dalmine BG, Italy}
\email{luca.brandolini@unibg.it}
\email{biancamaria.gariboldi@unibg.it}
\email{giacomo.gigante@unibg.it}
\email{alessandro.monguzzi@unibg.it}
\author[L. Colzani]{L. Colzani$^2$}
\address{\phantom{1}$^2$Dipartimento di Matematica e Applicazioni, Universit\`a degli
Studi di Milano Bicocca, Via R. Cozzi 55,  20125, Milano, Italy}
\email{leonardo.colzani@unimib.it}
\author[B. Gariboldi]{B. Gariboldi$^{1,2}$}
\author[G. Gigante]{G. Gigante$^1$}
\author[A. Monguzzi]{A. Monguzzi$^1$}
\thanks{\emph{Math Subject Classification 2020}: 11B68, 65B15, 42B05}
\thanks{The authors are members of Gruppo Nazionale per l'Analisi Matematica, la
Probabilit\`a e le loro Applicazioni (GNAMPA) of Istituto Nazionale di
Alta Matematica (INdAM)}
\thanks{The fifth author was partially supported by the Hellenic Foundation for Research and Innovation (H.F.R.I.) under the ``2nd Call for H.F.R.I. Research Projects to support Faculty Members \& Researchers'' (Project Number: 4662).}
\date{}
\keywords{Euler--MacLaurin summation formula, Bernoulli polynomials}

\begin{abstract}
We provide a multidimensional weighted Euler--MacLaurin summation formula on
polytopes and a multidimensional generalization of a result due to L. J.
Mordell on the series expansion in Bernoulli polynomials. These results are
consequences of a more general series expansion; namely, if $\chi
_{\tau\mathcal{P}}$ denotes the characteristic function of a dilated integer
convex polytope $\mathcal{P}$ and $q$ is a function with suitable
regularity, we prove that the periodization of $q\chi_{\tau\mathcal{P}}$
admits an expansion in terms of multivariate Bernoulli polynomials. These
multivariate polynomials are related to the Lerch Zeta function. In order to
prove our results we need to carefully study the asymptotic expansion of
$\widehat{q\chi_{\tau\mathcal{P}}}$, the Fourier transform of $q\chi
_{\tau\mathcal{P}}$.


\end{abstract}
\maketitle

\section{Introduction\label{intro}}

Our goal is to generalize to higher dimensions a result due to L.
J. Mordell and to deduce from this generalization a multidimensional weighted version of the classical Euler--MacLaurin formula and an associated quadrature rule. We first recall these results in the one dimensional setting. In order to do so we introduce the classical Bernoulli polynomials; there are
two possible normalizations that differ by a $n!$ factor and we use the following one.

\begin{definition}
The periodized Bernoulli polynomials $\{B_{n}\}_{n\in\mathbb{N }}$ are the
periodic functions that in the interval $(0,1)$ are defined recursively by the conditions
\begin{align*}
B_{0}(x)  &  =1,\qquad \dfrac{d}{dx}B_{n+1}(x)   =B_{n}(x),\qquad \int_{0}^{1}B_{n+1}(x)dx    =0.
\end{align*}
The value of these periodized functions when $x$ is an integer is given by%
\[
B_{n}(x) =\lim\limits_{\varepsilon\rightarrow0+}\frac{B_{n}(x+\varepsilon)
+B_{n}(x-\varepsilon) }{2}.
\]

\end{definition}

Mordell's theorem reads as follows.

\begin{theorem}
[Mordell 1966 \cite{Mor2, Mor}]\label{Thm Mordell-intro} Let $a<b$ and set
\[
\omega_{\lbrack a,b]}(x)=%
\begin{cases}
0 & x<a\text{ or }x>b\text{,}\\
1 & a<x<b\text{,}\\
1/2 & x=a\text{ or }x=b\text{.}%
\end{cases}
\]
(i) If $q\in C^{w+1}(\mathbb{R})$ then, for every $x\in\mathbb{R}$,
\[
\begin{split}
 \sum_{n=-\infty}^{+\infty}\omega_{\lbrack a,b]}(x+n)q(x+n)&\!=\!\!\int_{a}%
^{b}q(y)dy+\sum_{j=0}^{+w}\Big(  \dfrac{d^{j}q}{dx^{j}}(b)B_{j+1}%
(x-b)-\dfrac{d^{j}q}{dx^{j}}(a)B_{j+1}(x-a)\Big)\\
 &-\int_{a}^{b}\dfrac{d^{w+1}q}{dy^{w+1}%
 }(y)B_{w+1}(x-y)dy.
 \end{split}
\]
(ii) If $q\in C^{\infty}(\mathbb{R})$ and%
\[
\lim_{w\rightarrow+\infty}\left(  \frac{1}{2\pi}\right)  ^{w}\int_{a}%
^{b}\left\vert \dfrac{d^{w+1}q}{dy^{w+1}}(y)\right\vert dy=0,
\]
then, for every $x\in\mathbb{R}$,
\[
\sum_{n=-\infty}^{+\infty}\omega_{\lbrack a,b]}(x+n)q(x+n)\!=\!\int_{a}%
^{b}q(y)dy+\sum_{j=0}^{+\infty}\Big(  \dfrac{d^{j}q}{dx^{j}}(b)B_{j+1}%
(x-b)-\dfrac{d^{j}q}{dx^{j}}(a)B_{j+1}(x-a)\Big)  .
\]

\end{theorem}

Mordell's original result is essentially $(ii)$ above with $a=0$ and
$b=1$. Observe that the assumption on the growth of the derivatives of $q(x)$
implies that this function can be analytically extended to the entire complex
plane. The example $q( x) =\cos(2\pi x)$, which is $1$-periodic and has
expansion zero, shows that this assumption is sharp. Variants of this theorem
seem to be prior to Mordell's work (see e.g. \cite{BoasBuck}).

An immediate application of Theorem \ref{Thm Mordell-intro} is the classical
Euler-MacLaurin summation formula. Indeed, from $(i)$, when $a,b\in\mathbb{Z}$
and $x=0$, since $B_{j+1}(0)=0$ for even values of $j$, we obtain, for $q\in
C^{w+1}\left(  \mathbb{R}\right)  $,%
\begin{align*}
&  \bigg\vert \frac{1}{2}q(a)+q(a+1)+\ldots+q(b-1)+\frac{1}{2}q(b) \\
&  \qquad\qquad-  \int_{a}^{b}q(y)dy-\sum_{j=1}^{\left\lfloor \left(
w+1\right)  /2\right\rfloor }\left(  \dfrac{d^{2j-1}q}{dx^{2j-1}}%
(b)-\dfrac{d^{2j-1}q}{dx^{2j-1}}(a)\right)  B_{2j}(0)\bigg\vert \\
&  \leq\frac{\pi}{6}\left(  \frac{1}{2\pi}\right)  ^{w}\int_{a}^{b}\left\vert
\dfrac{d^{w+1}q}{dy^{w+1}}(y)\right\vert dy.
\end{align*}

It is well known that the above Euler-MacLaurin formula provides a quadrature
rule. Indeed, setting $a=0$, $b=N\in\mathbb Z^+$ and $q(x)=f\left(
x/N\right)  /N$ with $f\in C^{w+1}(\mathbb{R})$, one obtains%
\begin{align*}
\int_{0}^{1}f(y)dy  &  =\frac{1}{N}\left(  \frac{1}{2}f(0)+f\left(  \frac
{1}{N}\right)  +\ldots+f\left(  \frac{N-1}{N}\right)  +\frac{1}{2}f(1)\right)
\\
&\quad  +\sum_{j=1}^{\left\lfloor w/2\right\rfloor }\frac{1}{N^{2j}}\left(
\dfrac{d^{2j-1}f}{dx^{2j-1}}(0)-\dfrac{d^{2j-1}f}{dx^{2j-1}}(1)\right)
B_{2j}(0)+O(N^{-w-1}).
\end{align*}

Notice that only even powers of $N$ appear in the remainder terms.

To state our results in the multidimensional setting we need to introduce a number of definitions.

\begin{definition}
Let $\mathcal{P}$ be a measurable subset in $\mathbb{R}^{d}$. For every
$x\in\mathbb{R}^{d}$ the normalized solid angle at $x$ is given by
\[
\omega_{\mathcal{P}}(x)=\lim_{\varepsilon\rightarrow0+}\frac{1}{\left\vert
\left\{  |y|\leqslant1\right\}  \right\vert }\int_{|y|\leqslant1}%
\chi_{\mathcal{P}}(x-\varepsilon y)dy.
\]
Assuming that the above limit exists for every $x\in\mathbb{R}^{d}$, then, for every
continuous function $f\left(  x\right)  $ and for every positive integer $N$,
we set%
\[
S_{N}(f,\mathcal{P})=N^{-d}\sum_{n\in\mathbb{Z}^{d}}\omega_{\mathcal{P}%
}(N^{-1}n)f(N^{-1}n).
\]

\end{definition}

When $\mathcal{P}$ is a convex polytope (the convex hull of a finite number of
points) the weight $\omega_{\mathcal{P}}\left(  x\right)  $ is well defined for
every $x\in\mathbb{R}^{d}$. When $d=3$ the value of $\omega_{\mathcal{P}}(x)$
can be computed explicitly from the coordinates of the vertices of the
polytope using standard formulas of spherical trigonometry. See e.g.
\cite{Eriksson}. When $d>3$ see \cite{Aomoto}, \cite{BeckRobinsSam} and
\cite{Ribando}. \smallskip

These weights $\omega_{\mathcal{P}}(x)$ and weighted sums $S_{N}%
(q,\mathcal{P})$ are not new in the literature; for example MacDonald showed
that if $\mathcal{P}$ is a convex integer polytope (that is a convex polytope
with integer vertices) and $\tau$ is an integer dilation, then
\[
\sum_{n\in\mathbb{Z}^{d}}\omega_{\tau\mathcal{P}}(n)=({\mathrm{vol}%
}\mathcal{P})\tau^{d}+a_{d-2}\tau^{d-2}+\ldots+%
\begin{cases}
a_{1}\tau & \text{if $d$ is odd,}\\
a_{2}\tau^{2} & \text{if $d$ is even.}%
\end{cases}
\]
See e.g. \cite{BR} and \cite{DLR}.

An important property of these weights is that they are additive with respect
to $\mathcal{P}$. More precisely, if $\mathcal{P}_{1}$ and $\mathcal{P}_{2}$
have disjoint interior then%
\[
\omega_{\mathcal{P}_{1}\cup\mathcal{P}_{2}}(x)=\omega_{\mathcal{P}_{1}%
}(x)+\omega_{\mathcal{P}_{2}}(x).
\]
This implies that also the weighted Riemann sums $S_{N}(f,\mathcal{P})$ are
additive,
\[
S_{N}(f,\mathcal{P}_{1}\cup\mathcal{P}_{2})=S_{N}(f,\mathcal{P}_{1}%
)+S_{N}(f,\mathcal{P}_{2}).
\]
On the contrary, a different choice of weights may not guarantee the additivity.

\begin{definition}
\label{def:Bernoulli poly} For every multi-index of non-negative integers $J=(
j_{1},j_{2},...,j_{d}) $ and every $x=( x_{1},x_{2},\ldots,x_{d}) $ in
$\mathbb{R}^{d}$, define the multivariate Bernoulli polynomials
\[
B_{J}( x) =
\begin{cases}
B_{j_{1}}( x_{1}) B_{j_{2}}( x_{2}) \cdots B_{j_{d}}( x_{d}) &
\text{\textit{if }}0\leq x_{k}<1\text{\textit{,}}\\
0 & \text{\textit{otherwise.}}%
\end{cases}
\]
Moreover, for $L\in GL( d,\mathbb{Z}) $, define%
\[
B_{J,L}( x) =\vert L\vert^{-1}B_{J}\left(  ( L^{-1}) ^{t}x\right)  .
\]
Finally, define the periodized Bernoulli polynomials%
\[
\mathfrak{B}_{J,L}( x) =\sum_{n\in\mathbb{Z}^{d}}B_{J,L}( x+n) .
\]
At the points of discontinuity we assume the periodized Bernoulli
polynomials to be regularized so that%
\[
\mathfrak{B}_{J,L}( x) =\lim_{\varepsilon\to0+}\frac{1}{\left\vert \left\{
\vert y\vert\leqslant1\right\}  \right\vert }\int_{\vert y\vert\leqslant
1}\mathfrak{B}_{J,L}( x-\varepsilon y) dy.
\]

\end{definition}

We refer the reader to Section \ref{section-B} (Appendix B) for more comments on the construction of the periodized multivariate Bernoulli polynomials and their
connections with the Lerch Zeta functions.

\medskip

The next definitions are more technical and will be needed to describe the
asymptotic behavior along different directions of the Fourier transform of
functions supported on a simplex.

\begin{definition}
\label{bases-defn} For every dimension $d\geqslant1$, $\mathcal{F}^{(d)}$ is a
collection of $2^{d-1}$ bases of $\mathbb{R}^{d}$,%
\[
\mathcal{F}^{\left(  d\right)  }=\left\{  \mathcal{B}_{1}^{(d)},\ldots
,\mathcal{B}_{2^{d-1}}^{(d)}\right\}  .
\]
Each basis $\mathcal{B}_{j}^{(d)}$ consists of the vectors $b_{j,k}^{(d)}$,%
\[
\mathcal{B}_{j}^{(d)}=\left\{  b_{j,1}^{(d)},\dots,b_{j,d}^{(d)}\right\}  .
\]
The vectors $b_{j,k}^{(d)}$ are defined recursively as follows. If $d=1$, set
$b_{1,1}^{(1)}=1$. If $d=2$, set
\[
b_{1,1}^{(2)}=(1,0),~~b_{1,2}^{(2)}=(0,1)
\]
and
\[
b_{2,1}^{(2)}=(1,-1),~~b_{2,2}^{(2)}=(0,1).
\]
More generally, for $d\geqslant2$,%
\[
\mathcal{F}^{\left(  d\right)  }=\mathcal{F}_{1}^{\left(  d\right)  }%
\cup\mathcal{F}_{2}^{\left(  d\right)  }%
\]
where%
\[
\mathcal{F}_{1}^{\left(  d\right)  }=\left\{  \mathcal{B}_{1}^{(d)}%
,\ldots,\mathcal{B}_{2^{d-2}}^{(d)}\right\}  ,
\]%
\[
\mathcal{F}_{2}^{\left(  d\right)  }=\left\{  \mathcal{B}_{2^{d-2}+1}%
^{(d)},\ldots,\mathcal{B}_{2^{d-1}}^{(d)}\right\}
\]
and for $1\leqslant j\leqslant2^{d-2}$ we set%
\begin{align*}
b_{j,k}^{(d)}  &  =\left(  b_{j,k}^{(d-1)},0\right)  ,~~k=1,\ldots,d-1,\\
b_{j,d}^{(d)}  &  =(0,\ldots,0,1),
\end{align*}
and for $2^{d-2}+1\leqslant j\leqslant2^{d-1}$ we set%
\begin{align*}
b_{j,k}^{(d)}  &  =\left(  b_{j-2^{d-2},k}^{(d-1)},-b_{j-2^{d-2},k}%
^{(d-1)}\cdot\boldsymbol{1}_{d-1}\right)  ,~~k=1,\ldots,d-1,\\
b_{j,d}^{(d)}  &  =(0,\ldots,0,1),
\end{align*}
where $\boldsymbol{1}_{d-1}=\left(  1,\ldots,1\right)  \in\mathbb{R}^{d-1}$.
We will also associate to every basis $\mathcal{B\in F}^{\left(  d\right)  }$
with $d>1$ a $(d-1)$-dimensional multi-index $(v_{2},\ldots,v_{d})\in\left\{
1,2\right\}  ^{d-1}$ in the following way: $v_{d}=\ell$ if and only if
$\mathcal{B}\in\mathcal{F}_{\ell}^{(d)}$ and, if $d>2$, $(v_{2},\ldots
,v_{d-1})$ is the multi-index associated with the $(d-1)$-dimensional basis
$\mathcal{B}^{\prime}$ used to define $\mathcal{B}$ recursively. Observe that
there is a one to one correspondence between the bases in $\mathcal{F}_{\ell
}^{(d)}$ and the multi-indices in $\left\{  1,2\right\}  ^{d-1}$. Therefore,
given a multi-index $V=(v_{1},v_{2},\ldots,v_{d})\in\left\{  1,2\right\}
^{d}$, we will denote also by $\mathcal{B}_{V}$ the basis corresponding to the
vector $(v_{2},\ldots,v_{d})$.
\end{definition}

The role of $v_{1}$, the first component of the vector $V$, will be made clear
in what follows.

\begin{definition}
\label{lambda-defn} For every multi-index $V=(v_{1},\ldots,v_{d})\in\left\{
1,2\right\}  ^{d}$ we define the vectors $\lambda_{V}\in\mathbb{R}^{d}$
recursively as follows. For $d=1$%
\begin{align*}
\mathcal{\lambda}_{1}  &  =0,\\
\mathcal{\lambda}_{2}  &  =1.
\end{align*}
If $d=2$,%
\begin{align*}
\mathcal{\lambda}_{(1,1)}  &  =(0,0),\\
\mathcal{\lambda}_{(2,1)}  &  =(1,0),
\end{align*}
and%
\begin{align*}
\mathcal{\lambda}_{(1,2)}  &  =(0,1),\\
\mathcal{\lambda}_{(2,2)}  &  =(1,0).
\end{align*}
In general, for all $d\geqslant2$, if $v_{d}=1$ we set%
\[
\lambda_{(v_{1},v_{2},\ldots,v_{d})}=\left(  \lambda_{(v_{1},v_{2}%
,\ldots,v_{d-1})},0\right)  ,
\]
if $v_{d}=2$ we set%
\[
\lambda_{(v_{1},v_{2},\ldots,v_{d})}=\left(  \lambda_{(v_{1},v_{2}%
,\ldots,v_{d-1})},1-\lambda_{(v_{1},v_{2},\ldots,v_{d-1})}\cdot\mathbf{1}%
_{d-1}\right)  .
\]

\end{definition}

\medskip

\begin{center}
\begin{table}[h]%
\begin{tabular}
[c]{||c|r|r|l|l||}\hline
\rule{0pt}{3ex} $d$ & $V$ & $(v_{2},\ldots,v_{d})$ & $\mathcal{F}_{d}$ &
$\lambda_{V} $\\[0.5ex]\hline\hline
\rule{0pt}{4ex} 1 & 1 & / & $\mathcal{B}^{(1)}_{1}=\{1\}$ & $0$\\
\phantom{s} & 2 & \phantom{s} & \phantom{s} & $1$\\\hline\hline
\rule{0pt}{4ex} 2 & $(1,1)$ & $1$ & $\mathcal{B}^{(2)}_{1}=\{(1,0),(0,1)\}$ &
$(0,0)$\\
\phantom{s} & (2,1) &  & \phantom{s} & $(1,0)$\\\hline
\rule{0pt}{4ex} \phantom{s} & $(1,2)$ & $2$ & $\mathcal{B}^{(2)}%
_{2}=\{(1,-1),(0,1)\}$ & $(0,1)$\\
\phantom{s} & $(2,2)$ &  & \phantom{s} & $(1,0)$\\[1ex]\hline\hline
\rule{0pt}{4ex} 3 & $(1,1,1)$ & $(1,1)$ & $\mathcal{B}^{(3)}_{1}%
=\{(1,0,0),(0,1,0),(0,0,1)\}$ & $(0,0,0)$\\
\phantom{s} & $(2,1,1)$ &  & \phantom{s} & $(1,0,0)$\\
\phantom{s} & $(1,2,1)$ & $(2,1)$ & $\mathcal{B}^{(3)}_{2}%
=\{(1,-1,0),(0,1,0),(0,0,1)\}$ & $(0,1,0)$\\
\phantom{s} & $(2,2,1)$ &  & \phantom{s} & $(1,0,0)$\\\hline
\rule{0pt}{4ex} \phantom{s} & $(1,1,2)$ & $(1,2)$ & $\mathcal{B}^{(3)}%
_{3}=\{(1,0,-1),(0,1,-1),(0,0,1)\}$ & $(0,0,1)$\\
\phantom{s} & $(2,1,2)$ &  & \phantom{s} & $(1,0,0)$\\
\phantom{s} & $(1,2,2)$ & $(2,2)$ & $\mathcal{B}^{(3)}_{4}%
=\{(1,-1,0),(0,1,-1),(0,0,1)\}$ & $(0,1,0)$\\
\phantom{s} & $(2,2,2)$ &  & \phantom{s} & $(1,0,0)$\\[1ex]\hline
\end{tabular}
\medskip\caption{The various bases and multi-indices of Definitions
\ref{bases-defn} and \ref{lambda-defn}, for the dimensions $d=1,2,3$.}%
\label{ciao}%
\end{table}
\end{center}

We state our main result. Let $S_{d}\subseteq\mathbb{R}^{d}$ be the standard
simplex given by
\[
S_{d}=\left\{  x\in\mathbb{R}^{d}:x_{j}\geqslant0,\quad\sum_{j=1}^{d}%
x_{j}\leqslant1\right\}  .
\]

\begin{theorem}
\label{main-thm}Let $\mathcal{P}$ be a simplex in $\mathbb{R}^{d}$ with
vertices $\mathbf{0},\mathbf{m}_{1},\ldots,\mathbf{m}_{d}\in\mathbb{Z}^{d}$,
and let $M\in GL(d,\mathbb{Z})$ be the $d\times d$ matrix with columns
$\mathbf{m}_{1},\mathbf{m}_{2},\ldots,\mathbf{m}_{d}$, which maps the standard
simplex onto $\mathcal{P}$. Let $q\in C^{w+1}(\mathbb{R}^{d})$ with
$w\in\mathbb{N}$ and for $\tau>0$, let $q_{\tau,M}(x)=q(\tau Mx)$. Then, for
every $x\in\mathbb{R}^{d}$ and for every $\tau>0$,%
\begin{align*}
&  \sum_{n\in\mathbb{Z}^{d}}\omega_{\tau\mathcal{P}}(x+n)q(x+n)\\
=  &  \det(M)\sum_{V\in\{1,2\}^{d}}\sum_{I\in\{0,1\}^{d}}\sum_{|J|\leqslant
w,J\sqsubseteq I}\,\tau^{d-|I|-|J|}\left\langle \mu(V,I,J),q_{\tau
,M}\right\rangle \,\mathfrak{B}_{J+I,(MD_{V})^{t}}(x-\tau M\lambda_{V})\\
&  +\mathcal{R}_{w}(x).
\end{align*}
Here $\mu(V,I,J)$ are certain integro-differential functionals that will be
introduced in Definition \ref{Def mu}, $D_{V}=\left[  b_{1}|b_{2}|\cdots
|b_{d}\right]  $ where $\{b_{1},b_{2},\ldots,b_{d}\}$ is the basis
$\mathcal{B}_{V}$ and $J\sqsubseteq I$ means that $j_{k}=0$ if $i_{k}=0$.
Moreover, for every $\delta>0$ and every $\tau_{0}>0$ there exists a constant
$c$ depending on $\delta$ and $\tau_{0}$ but independent of $q$, $M$ and $w$,
such that for every $\tau>\tau_{0}$
\[
\left\vert \mathcal{R}_{w}(x)\right\vert \leqslant c\det(M)\tau^{d-w-1}%
(2^{d-2}\pi^{-1}+\delta)^{w+1}\sup_{w-d+2\leqslant|\alpha|\leqslant w+1}%
\sup_{x\in S_{d}}\left\vert \frac{\partial^{\alpha}q_{\tau, M}}{\partial
x^{\alpha}}(x)\right\vert .
\]

\end{theorem}

For $d=2$ a similar formula is contained in \cite{BCRT}. An immediate
consequence is the following corollary.

\begin{corollary}
\label{Coroll3} With the above notation, assume that $q\in C^{\infty
}(\mathbb{R}^{d})$ and that for some $c,\delta>0$,%
\[
\sup_{\vert\alpha\vert=w}\sup_{x\in S_{d}}\left\vert \frac{\partial^{\alpha
}q_{\tau, M}}{\partial x^{\alpha}}( x) \right\vert \leqslant c\tau^{w}(
2^{d-2}\pi^{-1}+\delta) ^{-w}.
\]
Then $\sum_{n\in\mathbb{Z}^{d}}\omega_{\tau\mathcal{P}}( x+n) q( x+n) $ can be
expanded in a uniformly convergent series of Bernoulli polynomials%
\begin{align*}
&  \sum_{n\in\mathbb{Z}^{d}}\omega_{\tau\mathcal{P}}( x+n) q( x+n)\\
=  &  \det(M) \sum_{V\in\{ 1,2\} ^{d}}\sum_{I\in\{ 0,1\} ^{d}}\sum_{\vert
J\vert\geq0,J\sqsubseteq I}\,\tau^{d-|I|-|J|}\left\langle \mu(V,I,J),q_{\tau,
M}\right\rangle \,\mathfrak{B}_{J+I,(MD_{V})^{t}}(x-\tau M\lambda_{V}).
\end{align*}

\end{corollary}

The uniform convergence in the above corollary seems paradoxical, since the
periodized function in the left-hand side is a priori discontinuous, but
observe that also in the right-hand side there are a priori infinitely many
Bernoulli polynomials that are discontinuous.

\smallskip Taking $x=0$ and $\tau\in\mathbb Z$, since the functions
$\mathfrak{B}_{J,L}(x)$ are periodic, from Theorem \ref{main-thm} one
immediately obtains an Euler--MacLaurin formula.

\begin{theorem}
\label{EM-d} Let $\mathcal{P}$ be a simplex in $\mathbb{R}^{d}$ with vertices
$\mathbf{0},\mathbf{m}_{1},\ldots,\mathbf{m}_{d}\in\mathbb{Z}^{d}$, and let
$M\in GL(d,\mathbb{Z})$ be the $d\times d$ matrix with columns $\mathbf{m}%
_{1},\mathbf{m}_{2},\ldots,\mathbf{m}_{d}$, which maps the standard simplex
onto $\mathcal{P}$. Let $q\in C^{w+1}(\mathbb{R}^{d})$ with $w\in\mathbb{N}$
and for $\tau>0$, let $q_{\tau, M}(x)=q(\tau Mx)$. Then, for every positive
integer $\tau>0$,
\begin{align*}
&  \sum_{n\in\mathbb{Z}^{d}}\omega_{\tau\mathcal{P}}(n)q(n)\\
=  &  \det(M)\sum_{V\in\{1,2\}^{d}}\sum_{I\in\{0,1\}^{d}}\sum_{|J|\leqslant
w,J\sqsubseteq I}\,\tau^{d-|I|-|J|}\left\langle \mu(V,I,J),q_{\tau
,M}\right\rangle \,\mathfrak{B}_{J+I,(MD_{V})^{t}}(0)+\mathcal{R}_{w}.
\end{align*}
Moreover, for every $\delta>0$ and every $\tau_{0}>0$ there exists a constant
$c$ depending on $\delta$ and $\tau_{0}$ but independent of $q$, $M$ and $w$,
such that for every $\tau>\tau_{0}$,%
\[
\left\vert \mathcal{R}_{w}\right\vert \leqslant c\det(M)\tau^{d-w-1}%
(2^{d-2}\pi^{-1}+\delta)^{w+1}\sup_{w-d+2\leqslant|\alpha|\leqslant w+1}%
\sup_{x\in S_{d}}\left\vert \frac{\partial^{\alpha}q_{\tau,M}}{\partial
x^{\alpha}}(x)\right\vert .
\]

\end{theorem}

We will see that when $I=(0,\ldots,0)$ the only non-vanishing term in the above sum corresponds to $V=(1,\ldots,1)$ and is the integral of $q$
over $\tau\mathcal{P}$.

\smallskip Similarly to the one dimensional case, Theorem \ref{EM-d} applied to the
function $N^{-d}f(x/N)$ gives a quadrature formula for simplices. Then, the
additivity of the weighted Riemann sums allows to extend this quadrature
formula to more general settings. Let us recall that a homogeneous simplicial
$d$-complex is a simplicial complex where every simplex of dimension less than
$d$ is a face of some simplex of dimension $d$. It is known that every
(bounded) convex polytope can be decomposed into simplices without additional
vertices. Hence, one can associate to a convex polytope a homogeneous
simplicial complex with the same vertices. This is obvious in dimension $d=2$,
less obvious in higher dimensions (see \cite{Edmonds}, see also Proposition
5.2 and Theorem 5.3 in \cite[Chapter 5]{RTR}).

\begin{theorem}
\label{Thm 1}Let $\mathcal{P}$ be a homogeneous simplicial $d$-complex with
integer vertices in $\mathbb{R}^{d}$. Let $w$ be a non-negative integer and
let $f\in C^{w+1}(\mathbb{R}^{d})$. Then, there exists a numerical sequence
$\left\{  \gamma_{k}\right\}  _{0<k\leqslant w/2}$ such that for every
positive integer $N$ we have%
\[
S_{N}(f,\mathcal{P})=\int_{\mathcal{P}}f(x)dx+\sum_{0<k\leqslant w/2}%
\gamma_{k}N^{-2k}+O(N^{-w-1}).
\]

\end{theorem}

For $d=2$ a similar formula is contained in \cite{BCRT}. A simple consequence
of Theorem \ref{Thm 1} is the following.

\begin{theorem}
\label{Thm 2}Let $\mathcal{P}$ be a homogeneous simplicial $d$-complex with
integer vertices in $\mathbb{R}^{d}$. Let $w$ be a non-negative integer and
let $f\in C^{w+1}(\mathbb{R}^{d})$. Finally, let $\left\{  c_{j}\right\}
_{0\leqslant j\leqslant w/2}$ be the solution of the Vandermonde system%
\[
\left[
\begin{array}
[c]{ccccc}%
1 & 1 & 1 & \cdots & 1\\
1 & 2^{-2} & (2^{-2})^{2} & \cdots & (2^{-2})^{\left\lfloor w/2\right\rfloor
}\\
1 & 2^{-4} & (2^{-4})^{2} & \cdots & (2^{-4})^{\left\lfloor w/2\right\rfloor
}\\
\vdots & \vdots & \vdots & \ddots & \vdots\\
1 & 2^{-2\left\lfloor w/2\right\rfloor } & (2^{-2\left\lfloor w/2\right\rfloor
})^{2} & \cdots & (2^{-2\left\lfloor w/2\right\rfloor })^{\left\lfloor
w/2\right\rfloor }%
\end{array}
\right]  \left[
\begin{array}
[c]{c}%
c_{0}\\
c_{1}\\
c_{2}\\
\vdots\\
c_{\left\lfloor w/2\right\rfloor }%
\end{array}
\right]  =\left[
\begin{array}
[c]{c}%
1\\
0\\
0\\
\vdots\\
0
\end{array}
\right]  .
\]
Then,
\[
\int_{\mathcal{P}}f(x)dx=\sum_{0\leqslant j\leqslant w/2}c_{j}S_{2^{j}%
N}(f,\mathcal{P})+O(N^{-w-1}).
\]

\end{theorem}

The coefficients $\gamma_{k}$ in Theorem \ref{Thm 1} are integro-differential
functionals applied to the function $f(x)$. In Theorem \ref{Thm 2} these
cumbersome coefficients have disappeared and only weighted Riemann sums are present.

We have not found a multidimensional analog of Mordell's theorem in the
literature. On the contrary the literature on multidimensional
Euler--MacLaurin summation formulas is vast and in continuous growth and to
have a comprehensive list of references is a challenging task. Here we recall
a few of these results and we try to compare them with ours, apologizing in
advance with all the authors that we do not explicitly mention.

If $\mathcal{P}$ denotes a regular integral convex polytope, Karshon, Sternberg and Weitsman obtained in \cite{KSW} the weighted formula%
\[
\sum_{n\in\mathbb{Z}^{d}}\sigma_{\mathcal{P}}(n)q(n)=\prod_{i=1}^{\ell}%
L^{2k}(D_{i})\int_{\mathcal{P}(h_{1},\ldots,h_{\ell})}q(x)\,dx+R_{\mathcal{P}%
}^{2k+1}(q)
\]
where $q\in C^{2k+1}$ is compactly supported, $R_{\mathcal{P}}^{2k+1}(f)$ is a
remainder explicitly given, $\ell$ is the number of facets, i.e., faces of
$\mathcal{P}$ of co-dimension $1$ and $\mathcal{P}(h_{1},\ldots,h_{\ell})$ is
a perturbation of the original polytope obtained expanding outward a distance
$h_{i}$ in the direction of th $i$-th facet. The weight function
$\sigma_{\mathcal{P}}$ is defined to be $0$ in the exterior of $\mathcal{P}$,
$1$ in the interior of $\mathcal{P}$ and $\sigma_{\mathcal{P}}(x)=2^{-c(x)}$
if $x$ is on the boundary of $\mathcal{P}$ and where $c(x)$ is the
co-dimension of the smallest face containing $x$. The operators $L^{2k}%
(D_{i})$ are the differential operators defined by the operators
$D_{i}=\partial/\partial h_{i}$, $i=1,\ldots,d$ and the functions
\[
L^{2k}(x)=1+\sum_{j=1}^{k}\frac{1}{(2j)!}b_{2j}x^{2j}%
\]
where the $b_{2j}$'s are Bernoulli numbers. A similar formula is proved for simple polytopes in \cite{KSW2}. Such Euler--MacLaurin formula is quite close
to our formula in the spirit, but we highlight a main difference. On one hand
the weight function $\sigma_{\mathcal{P}}$ is immediate to compute, since it
only depends on the co-dimension of a face at a given point. On the other hand
$\sigma_{\mathcal{P}}$ is not additive, whereas $\omega_{\mathcal{P}}$ is, and
this allows to apply Theorem \ref{EM-d} to polytopes by glueing simplices together.

We also refer the reader to the paper \cite{KSW3} and the references therein;
in this work the authors review and discuss the results in \cite{KSW, KSW2}
together with previous results by several different authors (\cite{KP,CS1,
CS2, BrionVergne}). See also \cite{agapito_weitsman}.

Another result we recall is the Euler-MacLaurin summation formula in
\cite{BV}. Let $\mathcal{P}\subseteq\mathbb{R}^{d}$ be a semi-rational convex
polyhedron of dimension $\ell\leq d$. Semi-rational means that the facets of
$\mathcal{P}$ are affine hyperplanes parallel to rational ones. Then, the
authors provide the asymptotic expansion, as $N\rightarrow+\infty$ ,
\[
\frac{1}{N^{\ell}}\sum_{n\in N\mathcal{P}\cap\mathbb{Z}^{d}}f(N^{-1}n)\sim
\int_{\mathcal{P}}f(x)dx+\sum_{k\geq1}a_{k}(N)N^{-k}.
\]
The authors also discuss their results in comparison with other previous
results (\cite{lFP, T}).

Finally, we also recall the works \cite{baldoni_berline_vergne, BV07,
fischer_pommersheim, GP, GS}.

\bigskip

Our proofs exploit harmonic analysis techniques with classical tools such as
the Poisson summation formula. Recall that if $f$ is an integrable function on $\mathbb R^d$
its Fourier transform $\widehat f$ is defined as
\[
\widehat{f}(\xi)=\int_{\mathbb{R}^{d}}f(x)e^{-2\pi i\xi\cdot x}dx.
\]

The following lemmas are well known.

\begin{lemma}
\label{Def Q-tilde}Let $\varphi(x)$ be a non-negative, radial, smooth,
function in $\mathbb{R}^{d}$, with compact support and integral one, and for
every $\varepsilon>0$ set $\varphi_{\varepsilon}(x)=\varepsilon^{-d}%
\varphi(\varepsilon^{-1}x)$. Let $\mathcal{P}$ be a convex polytope in
$\mathbb{R}^{d}$, let $q(x)$ be a smooth function in $\mathbb{R}^{d}$ and let
$Q(x)=q(x)\chi_{\mathcal{P}}(x)$. Then
\[
\lim_{\varepsilon\rightarrow0+}\varphi_{\varepsilon}\ast Q(x)=\omega
_{\mathcal{P}}(x)q(x).
\]

\end{lemma}

\begin{proof}
Integrate in polar coordinates.
\end{proof}

\begin{lemma}
\label{Poisson} With the notation of the above lemma, for every $\varepsilon
>0$ and every $x\in\mathbb{R}^{d}$ one has%
\[
\sum_{k\in\mathbb{Z}^{d}}\varphi_{\varepsilon}\ast Q(x+k)=\sum_{k\in
\mathbb{Z}^{d}}\widehat{\varphi}( \varepsilon k) \widehat{Q}( k) e^{2\pi
ik\cdot x}.
\]
Moreover%
\[
\sum_{k\in\mathbb{Z}^{d}}\omega_{\mathcal{P}}(x+k)q(x+k)=\lim_{\varepsilon
\to0+}\sum_{k\in\mathbb{Z}^{d}}\widehat{\varphi}( \varepsilon k) \widehat{Q}(
k) e^{2\pi ik\cdot x}.
\]
The first series is a finite sum. The second series converges absolutely.
\end{lemma}

\begin{proof}
This is the Poisson summation formula.
\end{proof}

To prove our results we need an explicit formula for the asymptotic expansion
of $\widehat{q\chi_{\mathcal{P}}}$ when $\mathcal{P}$ is a simplex, which
requires a non-trivial effort to be proved (Lemma
\ref{Lemma trasf simplesso standard} and Lemma
\ref{Lemma trasf simplesso generico}). The 2-dimensional case was dealt with
in \cite[Lemma 5]{BCRT}. Very elegant expansion formulas for $\widehat{\chi
_{\mathcal{P}}}$ (that is when $q\equiv 1$) in any dimension $d$ already appeared in \cite{Brion, DLR}. See also \cite{Sinai}
and the references therein.

The paper is organized as follows. In Section \ref{section-mordell} we present
a Fourier analytic proof of Mordell's theorem both for sake of completeness
and for illustrating the proof strategy that we will use in the multivariate
setting. In Section \ref{section-standard} we study the Fourier transform of a
function supported on a simplex. In Section \ref{Sect main res} we prove our
main result on the expansion of%
\[
\sum_{n\in\mathbb{Z}^{d}}\omega_{\tau\mathcal{P}}(x+n)q(x+n)
\]
in terms of our multivariate Bernoulli polynomials, that is Theorem
\ref{main-thm} and Corollary \ref{Coroll3}, whereas in Section
\ref{Sect proofs} we prove Theorem \ref{Thm 1} and Theorem \ref{Thm 2}. We
also include Appendix A (Section \ref{section-appendix}), where we collect some
well-known results in harmonic analysis on groups that we use,and Appendix B
(Section \ref{section-B}), where a further description of the periodized
multivariate Bernoulli polynomials is given.

\section{Bernoulli polynomials and a theorem of Mordell\label{section-mordell}%
}

The classical Bernoulli polynomials have elegant trigonometric expansions,
which predate Fourier. Recall that if $f$ is an integrable function on the
torus $\mathbb{T}=\mathbb{R}/\mathbb{Z}$, its Fourier series at a point $\xi$
is given by
\[
\sum_{n\in\mathbb{Z}}\widehat{f}(n)e^{2\pi in\xi}%
\]
where the Fourier coefficiente $\widehat{f}(n)$ is defined as
\[
\widehat{f}(n)=\int_{0}^{1}f(x)e^{-2\pi inx}\,dx.
\]

\begin{theorem}
[L. Euler, 1752]\textit{If }$n\geqslant1$, \textit{then, for every }$x$,%
\[
B_{n}(x)=-\sum_{k\in\mathbb{Z}\setminus\{ 0\} }\dfrac{e^{2\pi ikx}}{( 2\pi ik)
^{n}}.
\]

\end{theorem}

\begin{proof}
If $n=1$ and $0<x<1$, then 
\[
B_{1}(x)=x-1/2=\sum_{k=-\infty}^{+\infty}\left(  \int_{0}^{1}( y-1/2) e^{-2\pi
iky}dy\right)  e^{2\pi ikx}=-\sum_{k\in\mathbb{Z}\setminus\{ 0\} }%
\dfrac{e^{2\pi ikx}}{2\pi ik}.
\]
The symmetric partial sums with $-K\leqslant k\leqslant K$ of the above series
converge pointwise for every $0<x<1$, and by symmetry they also converge to
zero for $x=0$ and for $x=1$. Since $\frac{d}{dx}B_{n+1}(x)=B_{n}(x)$, the
Fourier expansion of $B_{n+1}(x)$ follows by integrating term by term the
series of $B_{n}(x)$. Since $\int_{0}^{1}B_{n+1}(x)dx=0$ the constant of
integration is zero. Observe that for $n>1$ the Fourier series of $B_{n}(x)$
converges absolutely and uniformly.
\end{proof}

The original proof of Euler is different and very interesting, see \cite{Eul}.
The following bounds are a consequence of the above trigonometric expansions.

\begin{corollary}
\label{Lemma Bernoulli}The periodic Bernoulli polynomials $B_{n}( x) $ with
$n\geqslant0$ are bounded by $( \pi^{2}/3) ( 2\pi) ^{-n}$. More precisely,%
\[
( 2\pi) ^{-n}\leqslant\sup_{x\in[ 0,1] }\vert B_{n}( x) \vert\leqslant(
\pi^{2}/3) ( 2\pi) ^{-n}.
\]

\end{corollary}

\begin{proof}
If $n=0$ then $B_{0}(x)=1$ and the lemma holds. If $n=1$ and $\left\lfloor
x\right\rfloor $ denotes the integer part of $x$, then%
\[
B_{1}(x)=x-\left\lfloor x\right\rfloor -1/2,
\]
so that%
\[
\sup_{x\in\lbrack0,1]}\left\vert B_{1}(x)\right\vert =1/2
\]
and again the lemma holds. Finally, if $n>1$,%
\[
|B_{n}(x)|=\left\vert \sum_{k\in\mathbb{Z}\setminus\{0\}}\frac{e^{2\pi ikx}%
}{(2\pi ik)^{n}}\right\vert \leqslant2(2\pi)^{-n}\sum_{k=1}^{+\infty}%
k^{-n}=2(2\pi)^{-n}\zeta(n).
\]
Observe that $\zeta(n)\leqslant\zeta(2)\leqslant\pi^{2}/6$. Also observe that
the Fourier coefficient with $k=1$ is $(2\pi i)^{-n}$, so that
\[
(2\pi)^{-n}\leqslant\int_{0}^{1}|B_{n}(x)|dz\leqslant\sup_{x\in\lbrack
0,1]}|B_{n}(x)|.
\]

\end{proof}

The following lemma provides an asymptotic expansion of the Fourier transform
of a piecewise smooth function.

\begin{lemma}
\label{Lemma1d}Let $w\geqslant0$. If the function $q(x)$ has $w+1$ integrable
derivatives in $[a,b]$, then for every $\xi\not =0$
\begin{align*}
\int_{a}^{b}q(x)e^{-2\pi ix\xi}dx=  &  \sum_{j=0}^{w}(2\pi i\xi)^{-j-1}\left(
e^{-2\pi ia\xi}\dfrac{d^{j}q}{dx^{j}}(a)-e^{-2\pi ib\xi}\dfrac{d^{j}q}{dx^{j}%
}(b)\right) \\
&  +(2\pi i\xi)^{-w-1}\int_{a}^{b}\dfrac{d^{w+1}q}{dx^{w+1}}(x)e^{-2\pi ix\xi
}dx.
\end{align*}

\end{lemma}

\begin{proof}
Integrate by parts.
\end{proof}

With the above results one easily obtains the following.

\begin{proof}
[Proof of Theorem \ref{Thm Mordell-intro}]By Lemma \ref{Lemma1d} and with the
notation of Lemma \ref{Def Q-tilde} and Lemma \ref{Poisson}, for every
$x\in\mathbb{R}$ we have the chain of equalities
\begin{align*}
&  \sum_{n=-\infty}^{+\infty}\omega_{\lbrack a,b]}(x+n)q(x+n)=\lim
_{\varepsilon\rightarrow0+}\sum_{k=-\infty}^{+\infty}\widehat{\varphi
}(\varepsilon k)\widehat{Q}(k)e^{2\pi ikx}\\
=  &  \lim_{\varepsilon\rightarrow0+}\sum_{k=-\infty}^{+\infty}\widehat
{\varphi}(\varepsilon k)\left(  \int_{a}^{b}q(y)e^{-2\pi iky}dy\right)
e^{2\pi ikx}\\
=  &  \int_{a}^{b}q(y)dy+\!\lim_{\varepsilon\rightarrow0+}\sum_{k\neq0}%
\widehat{\varphi}(\varepsilon k)\bigg(  \sum_{j=0}^{w}(2\pi ik)^{-j-1}\Big(
e^{-2\pi iak}\dfrac{d^{j}q}{dx^{j}}(a)-e^{-2\pi ibk}\dfrac{d^{j}q}{dx^{j}%
}(b)\Big)  \bigg)  e^{2\pi ikx}\\
&  +\lim_{\varepsilon\rightarrow0+}\sum_{k\neq0}\widehat{\varphi}(\varepsilon
k)\left(  (2\pi ik)^{-w-1}\int_{a}^{b}\dfrac{d^{w+1}q}{dy^{w+1}}(y)e^{-2\pi
iyk}dy\right)  e^{2\pi ikx}\\
=  &  \int_{a}^{b}q(y)dy+\sum_{j=0}^{w}\dfrac{d^{j}q}{dx^{j}}(a)\left(
\lim_{\varepsilon\rightarrow0+}\sum_{k\neq0}\widehat{\varphi}(\varepsilon
k)(2\pi ik)^{-j-1}e^{2\pi ik(x-a)}\right) \\
&  -\sum_{j=0}^{w}\dfrac{d^{j}q}{dx^{j}}(b)\left(  \lim_{\varepsilon
\rightarrow0+}\sum_{k\neq0}\widehat{\varphi}(\varepsilon k)(2\pi
ik)^{-j-1}e^{2\pi ik(x-b)}\right) \\
&  +\int_{a}^{b}\dfrac{d^{w+1}q}{dy^{w+1}}(y)\left(  \lim_{\varepsilon
\rightarrow0+}\sum_{k\neq0}\widehat{\varphi}(\varepsilon k)(2\pi
ik)^{-w-1}e^{2\pi ik(x-y)}\right)  dy\\
=  &  \int_{a}^{b}q(y)dy-\sum_{j=0}^{w}\dfrac{d^{j}q}{dx^{j}}(a)B_{j+1}%
(x-a)+\sum_{j=0}^{w}\dfrac{d^{j}q}{dx^{j}}(b)B_{j+1}(x-b)\\
&  -\int_{a}^{b}\dfrac{d^{w+1}q}{dy^{w+1}}(y)B_{w+1}(x-y)dy
\end{align*}
and $(i)$ is proved. The second part follows from Corollary
\ref{Lemma Bernoulli}.
\end{proof}

The above proof is not the original one of Mordell but it is inspired by
\cite{C-D-R} .

\section{The Fourier transform of a function supported on a simplex
\label{section-standard}}

A key ingredient for the proofs of our main results is a precise estimate of
the Fourier transform of a function restricted to a simplex. We first consider
the standard simplex.

\subsection{The standard simplex}

Let
\[
S_{d}=\left\{  x\in\mathbb{R}^{d}:x_{j}\geqslant0,\quad\sum_{j=1}^{d}%
x_{j}\leqslant1\right\}
\]
be the standard simplex. We want to give an asymptotic expansion of the
Fourier transform of the function $G( x) =g( x) \chi_{S_{d}}( x) $ where $g\in
C^{w+1}( \mathbb{R}^{d}) $. In \cite{Brion, DLR}, see also \cite{Sinai} and
the references therein, there are elegant symmetric formulas for
$\widehat{\chi_{S_{d}}}( \xi) $. The formulas we obtain are less elegant but
somehow more explicit and, in particular, we provide a formula when $\xi$
belongs to a singular direction as well. Since the asymptotic behaviour of
$\widehat{G}( \xi) $ depends on the faces of $S_{d}$ that are orthogonal to
$\xi$, it is natural to have different formulas in different regions. Hence,
we need to partition the space of frequencies into a finite number of cones
$\mathcal{Q}(\theta)$.

\begin{definition}
Let $\Theta_{d}$ be the class of all subspaces of $\mathbb{R}^{d}$ generated
by any possible choice of vectors in all the bases of $\mathcal{F}^{\left(
d\right)  }$. Then, $\Theta_{d}$ induces a partition of $\mathbb{R}^{d}$ into
a finite number of (possibly disconnected) conical regions $\mathcal{Q}%
(\theta)$, $\theta\in\Theta_{d}$, defined as follows: $\xi\in\mathcal{Q}%
(\theta)$ if and only if $\xi$ is orthogonal to all the vectors in $\theta$,
but it is not orthogonal to any other vector in the bases of $\mathcal{F}%
^{\left(  d\right)  }$ which is not in $\theta$. Namely,%
\begin{align*}
\mathcal{Q}(\theta)  &  =\left\{  \xi\in\mathbb{R}^{d}:\text{for all }%
b\in\bigcup\limits_{\mathcal{B\in F}^{\left(  d\right)  }}\mathcal{B}\text{,
}\xi\cdot b=0~\text{iff }b\in\theta\right\} \\
&  =\left\{  \xi\in\theta^{\perp}:\prod_{b\in\mathcal{B}\setminus\theta
}(b\cdot\xi)\neq0\text{, for all }\mathcal{B\in F}^{\left(  d\right)
}\right\}  .
\end{align*}

\end{definition}

We explicitly assume that the zero dimensional space belongs to $\Theta_{d}$
and in this case the associated cone has nonempty interior. In the other cases
such cones have empty interior.

\begin{lemma}
\label{Lemma Q(theta)}

(i) $\left\{  \mathcal{Q}(\theta)\right\}  _{\theta\in\Theta_{d}}$ is a
partition of $\mathbb{R}^{d}$.

(ii) Let $F(x)$ be a bounded function on $\mathbb{T}^{d}=\mathbb{R}%
^{d}/\mathbb{Z}^{d}$ and let $\theta$ be in $\Theta_{d}$. Then
\[
\sum_{m\in\mathcal{Q}(\theta)\cap\mathbb{Z}^{d}}\widehat{F}(m)e^{2\pi imx}%
\]
is a bounded function and there exists $c(\theta)$ such that%
\[
\sup_{x}\bigg\vert \sum_{m\in\mathcal{Q}(\theta)\cap\mathbb{Z}^{d}}\widehat
{F}(m)e^{2\pi imx}\bigg\vert \leqslant c(\theta)\sup_{x}\left\vert
F(x)\right\vert .
\]

\end{lemma}

\begin{proof}
$(i)$ For any $\xi\in\mathbb{R}^{d}$, $\xi\in\mathcal{Q}(\theta_{\xi})$ where
$\theta_{\xi}=\Big\langle b\in\bigcup\limits_{\mathcal{B\in F}^{\left(
d\right)  }}\mathcal{B}:b\cdot\xi=0\Big\rangle $. On the other hand if
$\theta_{1}\neq\theta_{2}$, that is if there exists, say, $v\in\theta
_{1}\setminus\theta_{2}$, then there exists $\overline{b}\in\bigcup
\limits_{\mathcal{B\in F}^{\left(  d\right)  }}\mathcal{B}$ such that
$\overline{b}\in\theta_{1}\setminus\theta_{2}$. Now if $\xi\in\mathcal{Q}%
(\theta_{1})$ then $\xi\cdot\overline{b}=0$. This implies that $\xi
\notin\mathcal{Q}(\theta_{2})$. Hence $\mathcal{Q}(\theta_{1})\cap
\mathcal{Q}(\theta_{2})=\emptyset$ and it follows that $\left\{
\mathcal{Q}(\theta)\right\}  _{\theta\in\Theta_{d}}$ is a partition of
$\mathbb{R}^{d}$.

$(ii)$ Let $\theta\in\Theta_{d}$, let $\theta^{\perp}=\left\{  \xi\in
\mathbb{R}^{d}:\xi\cdot b=0\text{ for every }b\in\theta\right\}  $ and for
$b\notin\theta$ let%
\[
\theta_{b}^{\perp}=\theta^{\perp}\cap\left\langle b\right\rangle ^{\perp}.
\]
Also $\{b_{1},\ldots,b_{N}\}=\Big\{  b\in\bigcup\limits_{\mathcal{B\in
F}^{\left(  d\right)  }}\mathcal{B}:b\notin\theta\Big\}  $. Then%
\[
\mathcal{Q}(\theta)=\theta^{\perp}\setminus\bigcup\limits_{j=1}^{N}%
\theta_{b_{j}}^{\perp}%
\]
and therefore%
\[
\sum_{m\in\mathcal{Q}(\theta)\cap\mathbb{Z}^{d}}\widehat{F}(m)e^{2\pi im\cdot
t}=\sum_{m\in\theta^{\perp}\cap\mathbb{Z}^{d}}\widehat{F}(m)e^{2\pi im\cdot
t}-\sum_{m\in%
{\textstyle\bigcup\nolimits_{j=1}^{N}}
L_{j}}\widehat{F}(m)e^{2\pi im\cdot t}%
\]
where $L_{j}=\theta_{b_{j}}^{\perp}\cap\mathbb{Z}^{d}$. By the
inclusion-exclusion principle%
\[
\chi_{%
{\textstyle\bigcup\nolimits_{j=1}^{N}}
L_{j}}(m)=\sum_{k=1}^{N}(-1)^{k-1}\sum_{\substack{I\subseteq\{1,2,\ldots
,N\}\\|I|=k}}\chi_{L_{I}}(m)
\]
where $L_{I}=%
{\displaystyle\bigcap\limits_{j\in I}}
L_{j}$. Therefore%
\[
\sum_{m\in\mathcal{Q}(\theta)\cap\mathbb{Z}^{d}}\widehat{F}(m)e^{2\pi im\cdot
t}=\sum_{m\in\theta^{\perp}\cap\mathbb{Z}^{d}}\widehat{F}(m)e^{2\pi im\cdot
t}+\sum_{k=1}^{N}(-1)^{k}\sum_{\substack{I\subseteq\{1,2,\ldots,N\}\\|I|=k}%
}\sum_{m\in L_{I}}\widehat{F}(m)e^{2\pi im\cdot t}.
\]
Observe now that $\theta^{\perp}\cap\mathbb{Z}^{d}$ and $L_{I}$ are subgroups
of $\mathbb{Z}^{d}$. To conclude the proof then it suffices to recall that the
restriction operator to a subgroup $\mathcal{H}$%
\[
R_{\mathcal{H}}F(t)=\sum_{m\in\mathcal{H}}\widehat{F}(m)e^{2\pi im\cdot t}%
\]
is a bounded operator on $L^{\infty}(\mathbb{T}^{d})$. See Lemma
\ref{annichilatore} in Appendix A.
\end{proof}

\medskip

We also need the following elementary lemma.

\begin{lemma}
\label{Lemma ovvio} (i) Set $n=( n^{\prime},n_{d}) \in\mathbb{Z}^{d-1}%
\times\mathbb{Z}$ and $(x^{\prime},x_{d})\in\mathbb{R}^{d-1}\times\mathbb{R}$.
Assume that $H( n^{\prime}) $ are the Fourier coefficients of a periodic
function $h( x^{\prime}) $ and $K( n_{d}) $ are the Fourier coefficients of a
periodic function $k( x_{d}) $. Then $H( n^{\prime}) K( n_{d}) $ are the
Fourier coefficients of $h( x^{\prime}) k( x_{d}) $.

(ii) Assume that $H( n) $ with $n\in\mathbb{Z}^{d}$ are the Fourier
coefficients of a periodic function bounded by $A$ and assume that $K( n_{d})
$ are the Fourier coefficients of a periodic function bounded by $B$. Then $H(
n) K( n_{d}) $ are the Fourier coefficients of a periodic function bounded by
$AB$.

(iii) Let $T\in SL( d, \mathbb{Z} ) $, $y\in\mathbb{R}^{d}$, and let $H( n) $
be the Fourier coefficients of a periodic function $h( x) $. Then $e^{2\pi
iy\cdot n}H( Tn) $ are the Fourier coefficients of the periodic function
$h\left(  ( T^{-1}) ^{t}( x+y) \right)  $.
\end{lemma}

\begin{proof}
The first part $(i)$ is trivial. To prove $(ii)$ let $h(x)$ be the periodic
function on $\mathbb{T}^{d}$ with Fourier coefficients $H(n)$ and $k(x_{d})$
be the periodic function on $\mathbb{T}$ with Fourier coefficients $K(n_{d})$.
Also, let $\mu$ be the product on the torus $\mathbb{T}^{d}$ of the Dirac
delta centered at the origin in the variables $x^{\prime}$ and $k(x_{d})$.
Then $\widehat{\mu}(n)=K(n_{d})$ and the total variation $\Vert\mu\Vert$ of
$\mu$ is bounded by $B$. Finally observe that $H(n)K(n_{d})$ are the Fourier
coefficients of $h\ast\mu(x)$ and
\[
|h\ast\mu(x)|\leqslant\sup|h(x)|\,\Vert\mu\Vert.
\]
Finally, the proof of $(iii)$ is very simple. If suffices to observe that%
\[
\sum_{n\in\mathbb{Z}^{d}}e^{2\pi iy\cdot n}H(Tn)e^{2\pi in\cdot x}=\sum
_{n\in\mathbb{Z}^{d}}H(Tn)e^{2\pi in\cdot(x+y)}=\sum_{m\in\mathbb{Z}^{d}%
}H(m)e^{2\pi im\cdot(T^{-1})^{t}(x+y)}.
\]

\end{proof}

With the notation introduced in Section \ref{intro}, we have the following crucial lemma.
\begin{lemma}
\label{Lemma trasf simplesso standard}Let $S_{d}$ be the standard simplex in
$\mathbb{R}^{d}$. There exist linear functionals $\{\alpha(\theta,V,J)\}$
indexed by $\theta\in\Theta_{d}$, $V\in\{1,2\}^{d}$, $J\in\mathbb{N}^{d}$ with
the following properties.

(i) For any integer $w\geqslant1$, for any $g\in C^{w+1}(\mathbb{R}^{d})$, for
every $\theta\in\Theta_{d}$ and for every $\xi\in\mathcal{Q}(\theta)$,%
\[
\widehat{g\chi_{S_{d}}}\left(  \xi\right)  =\sum_{V\in\{1,2\}^{d}}%
\sum_{|J|\leqslant w}\frac{\left\langle \alpha(\theta,V,J),g\right\rangle
e^{-2\pi i\lambda_{V}\cdot\xi}}{\prod_{b_{k}\in\mathcal{B}_{V}\setminus\theta
}(2\pi ib_{k}\cdot\xi)^{j_{k}+1}}+\mathcal{R}_{\theta,w}(g,\xi).
\]
In the above formula we adopt the following convention: $\mathcal{B}%
_{V}=\{b_{1},\ldots,b_{d}\}$ is the basis associated to the multi-index $V$
and $J=(j_{1},\ldots,j_{d})$ with $j_{k}=0$ whenever $b_{k}\in\theta$.

(ii) The coefficients $\left\langle \alpha(\theta,V,J),g\right\rangle $
satisfy the estimates%
\[
\left\vert \left\langle \alpha(\theta,V,J),g\right\rangle \right\vert
\leqslant c\,2^{(d-1)|J|}\sup_{|\alpha|\leqslant|J|}\sup_{x\in S_{d}%
}\left\vert \frac{\partial^{\alpha}g}{\partial x^{\alpha}}(x)\right\vert .
\]

(iii) The remainder $\mathcal{R}_{\theta,w}(g,\xi)$ has the property that for
every $\Omega>1/(2\pi)$ and every $\tau_{0}>0$ there exists $U=U(d)=U(d,\Omega
,\tau_{0})>0$ such that for every $\tau>\tau_{0}$ and $w\geqslant1$,%
\[
\left\{  \chi_{\mathcal{Q}(\theta)}(n)\mathcal{R}_{\theta,w}(g,\tau
n)\right\}  _{n\in\mathbb{Z}^{d}}%
\]
are the Fourier coefficients of a function on the torus $\mathbb{T}^{d}$
bounded by%
\[
U\,(2^{d-1}\Omega\tau^{-1})^{w+1}\sup_{w-d+2\leqslant|\alpha|\leqslant
w+1}\sup_{x\in S_{d}}\left\vert \frac{\partial^{\alpha}g}{\partial x^{\alpha}%
}(x)\right\vert .
\]

\end{lemma}

\begin{proof}
The proof is by induction on the dimension $d$. Let $G\left(  x\right)
=g\left(  x\right)  \chi_{S_{d}}\left(  x\right)  $. The case $d=1$ is covered
by Lemma \ref{Lemma1d}. Here we reinterpret the result using the formalism of
the cones $\mathcal{Q}(\theta)$. We have $\mathcal{F}^{\left(  1\right)
}=\{\mathcal{B}\}$ where $\mathcal{B}=\{1\}$ and $\Theta_{1}=\left\{
\{0\},\,\mathbb{R}\right\}  $. Let $\theta=\mathbb{R}$, then $\mathcal{Q}%
(\theta)=\{0\}$, and
\[
\widehat{G}(0)=\int_{0}^{1}g(x)dx,
\]
so that
\[
\left\langle \alpha\left(  \mathbb{R},V,j\right)  ,g\right\rangle =\left\{
\begin{array}
[c]{ll}%
\int_{0}^{1}g(x)dx & j=0,V=1,\\
0 & \text{otherwise,}%
\end{array}
\right.
\]
and the remainder $\mathcal{R}_{\mathbb{R},w}(g,0)=0$ for every $w$.

Let $\theta=\{0\}$, then $\mathcal{Q}(\theta)=\left\{  \xi\in\mathbb{R},\text{
}\xi\neq0\right\}  $. In this case, by Lemma \ref{Lemma1d},%
\begin{align*}
\int_{0}^{1}g(x)e^{-2\pi ix\xi}dx  &  =\sum_{j=0}^{w}\frac{\dfrac{d^{j}%
g}{dx^{j}}(0)-e^{-2\pi i\xi}\dfrac{d^{j}g}{dx^{j}}(1)}{(2\pi i\xi)^{j+1}}\\
&  \quad+(2\pi i\xi)^{-w-1}\int_{0}^{1}\dfrac{d^{w+1}g}{dx^{w+1}}(x)e^{-2\pi
ix\xi}dx.
\end{align*}
It follows that
\begin{align*}
\left\langle \alpha\left(  \{0\},1,j\right)  ,g\right\rangle  &  =\dfrac
{d^{j}g}{dx^{j}}(0),\\
\left\langle \alpha\left(  \{0\},2,j\right)  ,g\right\rangle  &
=-\dfrac{d^{j}g}{dx^{j}}(1).
\end{align*}
Moreover the remainder evaluated at the lattice points $\xi=\tau n$ is%
\[
\mathcal{R}_{\{0\},w}(g,\tau n)=(2\pi i\tau n)^{-w-1}\int_{0}^{1}%
\dfrac{d^{w+1}g}{dx^{w+1}}(x)e^{-2\pi i\tau nx}dx
\]
Hence for every $w\geqslant1$,%
\begin{align*}
&  \left\vert \sum_{n\in\mathbb{Z}}\mathcal{R}_{\{0\},w}(g,\tau n)\chi
_{\mathcal{Q}(\{0\})}(\tau n)e^{2\pi inx}\right\vert =\left\vert \sum_{n\neq
0}\mathcal{R}_{\{0\},w}(g,\tau n)e^{2\pi inx}\right\vert \\
\leqslant &  (2\pi\tau)^{-w-1}\left\{  \sum_{n\neq0}|n|^{-w-1}\right\}
\sup_{x\in\lbrack0,1]}\left\vert \dfrac{d^{w+1}g}{dx^{w+1}}(x)\right\vert \\
\leqslant &  (2\pi\tau)^{-w-1}\frac{\pi^{2}}{3}\sup_{x\in\lbrack
0,1]}\left\vert \dfrac{d^{w+1}g}{dx^{w+1}}(x)\right\vert .
\end{align*}
Let now $d\geqslant2$ and assume that the theorem holds in dimension $d-1$.
Fix $\theta\in\Theta_{d}$. Observe that the vector $e_{d}=(0,\ldots,0,1)$
belongs to at least one (actually all) bases in $\mathcal{F}^{\left(
d\right)  }$. If $e_{d}\in\theta$, then for all $\xi\in\mathcal{Q}(\theta)$
one has $\xi_{d}=0$. For this choice of $\theta$ and for all $\xi=(\xi
^{\prime},0)\in\mathcal{Q}(\theta)$ we have the following formula%
\begin{align*}
\widehat{G}(\xi)  &  =\int_{S_{d}}g(x)e^{-2\pi ix\cdot\xi}dx=\int_{S_{d-1}%
}e^{-2\pi ix^{\prime}\cdot\xi^{\prime}}\left[  \int_{0}^{1-(x_{1}+x_{2}%
+\cdots+x_{d-1})}g(x^{\prime},x_{d})dx_{d}\right]  dx^{\prime}\\
&  =\int_{S_{d-1}}e^{-2\pi ix^{\prime}\cdot\xi^{\prime}}F(x^{\prime
})dx^{\prime}.
\end{align*}
Observe that $\xi\in\mathcal{Q}(\theta)$ if and only if $\xi^{\prime}%
\in\mathcal{Q}(\theta^{\prime})$ where $\theta^{\prime}$ is the space of the
vectors $b^{\prime}$ such that $(b^{\prime},b_{d})\in\theta$ for some $b_{d}$.
Indeed, since $\xi\cdot b=\xi^{\prime}\cdot b^{\prime}$, $\xi^{\prime}$ is
orthogonal to $\theta^{\prime}$ if and only if $\xi$ is orthogonal to $\theta
$. Since $\theta^{\prime}\in\Theta_{d-1}$ (see Lemma \ref{Lemma 22} for details),we may therefore apply the
$(d-1)$-dimensional formula corresponding to $\theta^{\prime}$ to the function
$F(x^{\prime})$ so that%
\[
\widehat{G}(\xi)=\sum_{V\mathcal{^{\prime}}\in\{1,2\}^{d-1}}\sum_{|J^{\prime
}|\leqslant w}\frac{\left\langle \alpha(\theta^{\prime},V^{\prime},J^{\prime
}),F\right\rangle e^{-2\pi i\lambda_{V^{\prime}}\cdot\xi^{\prime}}}%
{\prod_{b_{k}^{\prime}\in\mathcal{B}_{V^{\prime}}\setminus\theta^{\prime}%
}(2\pi ib_{k}^{\prime}\cdot\xi^{\prime})^{j_{k}^{\prime}+1}}+\mathcal{R}%
_{\theta^{\prime},w}(F,\xi^{\prime}),
\]
where $j_{k}^{\prime}=0$ if $b_{k}^{\prime}\in\theta^{\prime}$. Observe that
this expression can be written in the form%
\[
\widehat{G}(\xi)=\sum_{V=(V^{\prime},1)\in\{1,2\}^{d}}\sum_{J=\left(
J^{\prime},0\right)  ,\text{ }|J|\leqslant w}\frac{\left\langle \alpha
(\theta^{\prime},V^{\prime},J^{\prime}),F\right\rangle e^{-2\pi i\lambda
_{V}\cdot\xi}}{\prod_{b_{k}\in\mathcal{B}_{V}\setminus\theta}(2\pi ib_{k}%
\cdot\xi)^{j_{k}+1}}+\mathcal{R}_{\theta^{\prime},w}(F,\xi^{\prime})
\]
where $j_{k}=0$ if $b_{k}\in\theta$. The coefficients $\left\langle
\alpha(\theta^{\prime},V^{\prime},J^{\prime}),F\right\rangle $ satisfy the
estimates%
\begin{align*}
\left\vert \left\langle \alpha(\theta^{\prime},V^{\prime},J^{\prime
}),F\right\rangle \right\vert  &  \leqslant c\,2^{(d-2)|J^{\prime}|}%
\sup_{|\alpha|\leqslant|J^{\prime}|}\sup_{x^{\prime}\in S_{d-1}}\left\vert
\left(  \frac{\partial}{\partial x^{\prime}}\right)  ^{\alpha}F(x^{\prime
})\right\vert \\
&  \leqslant c\,2^{(d-1)|J|}\sup_{|\alpha|\leqslant|J|}\sup_{x\in S_{d}%
}\left\vert \frac{\partial^{\alpha}g}{\partial^{\alpha}x}(x)\right\vert .
\end{align*}
Indeed, for every $1\leqslant j,k,\ldots\leqslant d-1$, we have%
\begin{align*}
\frac{\partial F}{\partial x_{j}}(x^{\prime})  &  =\frac{\partial}{\partial
x_{j}}\int_{0}^{1-x^{\prime}\cdot\mathbf{1}_{d-1}}g(x^{\prime},x_{d})dx_{d}\\
&  =-g(x^{\prime},1-x^{\prime}\cdot\mathbf{1}_{d-1})+\int_{0}^{1-x^{\prime
}\cdot\mathbf{1}_{d-1}}\frac{\partial g}{\partial x_{j}}(x^{\prime}%
,x_{d})dx_{d},\\
\frac{\partial^{2}F}{\partial x_{k}\partial x_{j}}(x^{\prime})  &
=-\frac{\partial g}{\partial x_{k}}(x^{\prime},1-x^{\prime}\cdot
\mathbf{1}_{d-1})+\frac{\partial g}{\partial x_{d}}(x^{\prime},1-x^{\prime
}\cdot\mathbf{1}_{d-1})\\
&  \quad-\frac{\partial g}{\partial x_{j}}(x^{\prime},1-x^{\prime}%
\cdot\mathbf{1}_{d-1})+\int_{0}^{1-x^{\prime}\cdot\mathbf{1}_{d-1}}%
\frac{\partial^{2}g}{\partial x_{k}\partial x_{j}}(x^{\prime},x_{d})dx_{d},
\end{align*}
and so on. Thus%
\begin{align*}
&  \left\langle \alpha\left(  \theta,V,J\right)  ,g\right\rangle \\
&  =\left\{
\begin{array}
[c]{ll}%
\left\langle \alpha\left(  \theta^{\prime},V^{\prime},J^{\prime}\right)
,\int_{0}^{1-x^{\prime}\cdot\mathbf{1}_{d-1}}g(x^{\prime},x_{d})dx_{d}%
\right\rangle  & \text{if }J=\left(  J^{\prime},0\right)  \text{ and
}V=\left(  V,1\right)  ,\\
0 & \text{otherwise.}%
\end{array}
\right.
\end{align*}
For $\xi\in\mathcal{Q}\left(  \theta\right)  $ set $\mathcal{R}%
_{\theta,w}(g,\xi)=\mathcal{R}_{\theta^{\prime},w}(F,\xi^{\prime})$ (recall
that $\xi=\left(  \xi^{\prime},0\right)  $) and observe that
\begin{align*}
\sum_{n\in\mathbb{Z}^{d}}\chi_{\mathcal{Q}(\theta)}(n)\mathcal{R}_{\theta
,w}(g,\tau n)e^{2\pi in\cdot x}  &  =\sum_{n^{\prime}\in\mathbb{Z}^{d-1}}%
\chi_{\mathcal{Q}(\theta)}(n^{\prime},0)\mathcal{R}_{\theta,w}\left(  g,(\tau
n^{\prime},0)\right)  e^{2\pi in^{\prime}\cdot x^{\prime}}\\
&  =\sum_{n^{\prime}\in\mathbb{Z}^{d-1}}\chi_{\mathcal{Q}(\theta^{\prime}%
)}(n^{\prime})\mathcal{R}_{\theta^{\prime},w}(F,\tau n^{\prime})e^{2\pi
in^{\prime}\cdot x^{\prime}}.
\end{align*}
By induction, $\left\{  \chi_{\mathcal{Q}(\theta^{\prime})}(n^{\prime
})\mathcal{R}_{\theta^{\prime},w}(F,\tau n^{\prime})\right\}  _{n^{\prime}%
\in\mathbb{Z}^{d-1}}$ are the Fourier coefficients of a function on the
$(d-1)$-dimensional torus bounded by%
\begin{align*}
&  U(d-1)(2^{d-2}\Omega\tau^{-1})^{w+1}\sup_{w-d+3\leqslant|\alpha|\leqslant
w+1}\sup_{x^{\prime}\in S_{d-1}}\left\vert \left(  \frac{\partial}{\partial
x^{\prime}}\right)  ^{\alpha}F(x^{\prime})\right\vert \\
&  \leqslant U(d-1)(2^{d-2}\Omega\tau^{-1})^{w+1}\sup_{w-d+2\leqslant
|\alpha|\leqslant w+1}2^{|\alpha|}\sup_{x\in S_{d}}\left\vert \left(
\frac{\partial}{\partial x}\right)  ^{\alpha}g(x)\right\vert \\
&  \leqslant U(d-1)(2^{d-1}\Omega\tau^{-1})^{w+1}\sup_{w-d+2\leqslant
|\alpha|\leqslant w+1}\sup_{x\in S_{d}}\left\vert \left(  \frac{\partial
}{\partial x}\right)  ^{\alpha}g(x)\right\vert .
\end{align*}
Hence, by $(i)$ in Lemma \ref{Lemma ovvio} with $K(n_{d})=0$ if $n_{d}\neq0$ and
$K(n_{d})=1$ if $n_{d}=0$ so that $k(x_{d})=1$, and $H(n^{\prime}%
)=\mathcal{R}_{w}(F,\tau n^{\prime})$ it follows that
\[
H(n^{\prime})K(n_{d})=\left\{
\begin{array}
[c]{cc}%
\chi_{\mathcal{Q}(\theta^{\prime})}(n^{\prime})\mathcal{R}_{\theta^{\prime}%
,w}(F,\tau n^{\prime}) & n_{d}=0,\\
0 & n_{d}\neq0
\end{array}
\right.
\]
are the Fourier coefficients of a function on the $d$-dimensional torus
bounded by%
\[
U(d-1)(2^{d-1}\Omega\tau^{-1})^{w+1}\sup_{w-d+2\leqslant|\alpha|\leqslant
w+1}\sup_{x\in S_{d}}\left\vert \left(  \frac{\partial}{\partial x}\right)
^{\alpha}g(x)\right\vert .
\]
Now assume that $e_{d}\notin\theta$, so that $\xi_{d}\neq0$ for all $\xi
\in\mathcal{Q}(\theta)$. Then, by Lemma \ref{Lemma1d},%
\begin{align*}
\widehat{G}(\xi)  &  =\int_{S_{d-1}}e^{-2\pi ix^{\prime}\cdot\xi^{\prime}%
}\left[  \int_{0}^{1-x^{\prime}\cdot\mathbf{1}_{d-1}}g(x^{\prime}%
,x_{d})e^{-2\pi ix_{d}\xi_{d}}dx_{d}\right]  dx^{\prime}\\
&  =\sum_{j_{d}=0}^{w}(2\pi i\xi_{d})^{-j_{d}-1}\int_{S_{d-1}}e^{-2\pi
ix^{\prime}\cdot\xi^{\prime}}\frac{\partial^{j_{d}}g}{\partial x_{d}^{j_{d}}%
}(x^{\prime},0)dx^{\prime}\\
&  \quad-\sum_{j_{d}=0}^{w}(2\pi i\xi_{d})^{-j_{d}-1}\int_{S_{d-1}}e^{-2\pi
ix^{\prime}\cdot\xi^{\prime}}e^{-2\pi i(1-x^{\prime}\cdot\mathbf{1}_{d-1}%
)\xi_{d}}\frac{\partial^{j_{d}}g}{\partial x_{d}^{j_{d}}}(x^{\prime
},1-x^{\prime}\cdot\mathbf{1}_{d-1})dx^{\prime}\\
&  \quad+(2\pi i\xi_{d})^{-w-1}\int_{S_{d}}\frac{\partial^{w+1}g}{\partial
x_{d}^{w+1}}(x)e^{-2\pi ix\cdot\xi}dx\\
&  =I+II+III.
\end{align*}
The term $III$ is part of the remainder. In view of Lemma \ref{Lemma Q(theta)}
and since $n\in\mathcal{Q}(\theta)\cap\mathbb{Z}^{d}$ implies $n_{d}\neq0$ it
suffices to show that%
\[
\sum_{n\in\mathbb{Z}^{d},n_{d}\neq0}\left[  (2\pi i\tau n_{d})^{-w-1}%
\int_{S_{d}}\frac{\partial^{w+1}g}{\partial x_{d}^{w+1}}(x)e^{-2\pi
ix\cdot\tau n}dx\right]  e^{2\pi inx}%
\]
is the Fourier series of a bounded function. Observe first that the integrals
in the above sum are the Fourier coefficients of a bounded function on the
torus,%
\begin{align*}
 \int_{S_{d}}\frac{\partial^{w+1}g}{\partial x_{d}^{w+1}}(x)&e^{-2\pi
ix\cdot\tau n}dx =\tau^{-d}\int_{\mathbb{R}^{d}}\chi_{S_{d}}(\tau^{-1}y)\frac{\partial
^{w+1}g}{\partial x_{d}^{w+1}}(\tau^{-1}y)e^{-2\pi iy\cdot n}dy\\
&  =\tau^{-d}\int_{\mathbb{T}^{d}}\left[  \sum_{k\in\mathbb{Z}^{d}}\chi
_{S_{d}}\left(  \tau^{-1}(y+k)\right)  \frac{\partial^{w+1}g}{\partial
x_{d}^{w+1}}\left(  \tau^{-1}(y+k)\right)  \right]  e^{-2\pi iy\cdot n}dy.
\end{align*}
The inner sum is finite and consists of at most $c(1+\tau)^{d}$ terms. By $(ii)$
in Lemma \ref{Lemma ovvio} and Corollary \ref{Lemma Bernoulli}, multiplying by
$(2\pi i\tau n_{d})^{-w-1}$, we obtain again the Fourier coefficients of a
periodic function bounded by%
\begin{align*}
&  (2\pi\tau)^{-w-1}\tau^{-d}\frac{\pi^{2}}{3}\sup_{y\in\mathbb{R}^{d}%
}\left\vert \sum_{k\in\mathbb{Z}^{d}}\chi_{S_{d}}\left(  \tau^{-1}%
(y+k)\right)  \frac{\partial^{w+1}g}{\partial x_{d}^{w+1}}\left(  \tau
^{-1}(y+k\right)  )\right\vert \\
&  \leqslant c\tau^{-d}\left(  1+\tau\right)  ^{d}(2\pi\tau)^{-w-1}\sup_{x\in
S_{d}}\left\vert \frac{\partial^{w+1}g}{\partial x_{d}^{w+1}}(x)\right\vert \\
&  \leqslant c(1+\tau_{0}^{-1})^{d}(2^{d-1}\Omega\tau^{-1})^{w+1}\sup_{x\in
S_{d}}\left\vert \frac{\partial^{w+1}g}{\partial x_{d}^{w+1}}(x)\right\vert
\end{align*}
with $c$ independent of $w$ and $g$ and $\tau\geqslant\tau_{0}>0$.

Let us consider the term $I$. Let $\theta_{1}=\{b^{\prime}:(b^{\prime}%
,0)\in\theta\}\in\Theta_{d-1}$ (see Lemma \ref{Lemma 22} for details). We claim that%
\[
\xi=\left(  \xi^{\prime},\xi_{d}\right)  \in\mathcal{Q}(\theta)\implies
\xi^{\prime}\in\mathcal{Q}(\theta_{1}).
\]
Indeed, let $\xi\in\mathcal{Q}(\theta)$ and let $b^{\prime}\in\theta_{1}%
\cap\Big(  \bigcup\limits_{\mathcal{B\in F}^{\left(  d-1\right)  }%
}\mathcal{B}\Big)  $. Then $(b^{\prime},0)\in\theta$ and therefore%
\[
0=\xi\cdot(b^{\prime},0)=\xi^{\prime}\cdot b^{\prime}.
\]
Now, let $b^{\prime}\in\bigcup\limits_{\mathcal{B\in F}^{\left(  d-1\right)
}}\mathcal{B}$, $b^{\prime}\notin\theta_{1}$, then $(b^{\prime},0)\notin
\theta$. Since $(b^{\prime},0)\in\bigcup\limits_{\mathcal{B\in F}^{\left(
d\right)  }}\mathcal{B}$ we have%
\[
0\not =\xi\cdot(b^{\prime},0)=\xi^{\prime}\cdot b^{\prime}.
\]
Applying the $d-1$ dimensional formula corresponding to $\theta_{1}$ to the
function $\frac{\partial^{j_{d}}g}{\partial x_{d}^{j_{d}}}(x^{\prime},0)$ we
obtain%
\begin{align*}
I=  &  \sum_{j_{d}=0}^{w}(2\pi i\xi_{d})^{-j_{d}-1}\int_{S_{d-1}}e^{-2\pi
ix^{\prime}\cdot\xi^{\prime}}\frac{\partial^{j_{d}}g}{\partial x_{d}^{j_{d}}%
}(x^{\prime},0)dx^{\prime}\\
=  &  \sum_{j_{d}=0}^{w}(2\pi i\xi_{d})^{-j_{d}-1}\sum_{V^{\prime}%
\in\{1,2\}^{d-1}}\sum_{|J^{\prime}|\leqslant w-j_{d}}\frac{\left\langle
\alpha(\theta_{1},V^{\prime},J^{\prime}),\left(  \partial/\partial
x_{d}\right)  ^{j_{d}}g(\cdot,0)\right\rangle e^{-2\pi i\lambda_{V^{\prime}%
}\cdot\xi^{\prime}}}{\prod_{b_{k}^{\prime}\in\mathcal{B}_{V^{\prime}}%
\setminus\theta_{1}}(2\pi ib_{k}^{\prime}\cdot\xi^{\prime})^{j_{k}^{^{\prime}%
}+1}}\\
&  +\sum_{j_{d}=0}^{w}(2\pi i\xi_{d})^{-j_{d}-1}\mathcal{R}_{\theta
_{1},w-j_{d}}\left(  \frac{\partial^{j_{d}}g}{\partial x_{d}^{j_{d}}}%
(\cdot,0),\xi^{\prime}\right) \\
=  &  \sum_{V=(V^{\prime},1)\in\{1,2\}^{d}}\sum_{|J|\leqslant w}%
\frac{\left\langle \alpha(\theta_{1},V^{\prime},J^{\prime}),\left(
\partial/\partial x_{d}\right)  ^{j_{d}}g(\cdot,0)\right\rangle e^{-2\pi
i\lambda_{V}\cdot\xi}}{\prod_{b_{k}\in\mathcal{B}_{V}\setminus\theta}(2\pi
ib_{k}\cdot\xi)^{j_{k}+1}}\\
&  +\sum_{j_{d}=0}^{w}(2\pi i\xi_{d})^{-j_{d}-1}\mathcal{R}_{\theta
_{1},w-j_{d}}\left(  \frac{\partial^{j_{d}}g}{\partial x_{d}^{j_{d}}}%
(\cdot,0),\xi^{\prime}\right)  .
\end{align*}
The double sum is part of the main term in the asymptotic expansion and the
last sum is part of the remainder. By induction, since $J=(J^{\prime},j_{d})$,%
\begin{align*}
\left\vert \left\langle \alpha(\theta_{1},V^{\prime},J^{\prime}),\left(
\partial/\partial x_{d}\right)  ^{j_{d}}g(\cdot,0)\right\rangle \right\vert
&  \leqslant c2^{(d-2)|J^{\prime}|}\sup_{|\alpha|\leqslant|J^{\prime}|}%
\sup_{x\in S_{d}}\left\vert \left(  \frac{\partial}{\partial x^{\prime}%
}\right)  ^{\alpha}\left(  \frac{\partial}{\partial x_{d}}\right)  ^{j_{d}%
}g(x)\right\vert \\
&  \leqslant c2^{(d-2)|J|}\sup_{|\alpha|\leqslant|J|}\sup_{x\in S_{d}%
}\left\vert \frac{\partial^{\alpha}g}{\partial x^{\alpha}}(x)\right\vert .
\end{align*}
We now deal with the remainder in $I$ as follows. By the induction assumption%
\[
\sum_{n^{\prime}\in\mathcal{Q}(\theta^{\prime})\cap\mathbb{Z}^{d-1}%
}\mathcal{R}_{\theta_{1},w-j_{d}}\left(  \frac{\partial^{j_{d}}g}{\partial
x_{d}^{j_{d}}}(\cdot,0),\tau n^{\prime}\right)  e^{2\pi in^{\prime}\cdot
x^{\prime}}%
\]
is the Fourier expansion of a function $F(x^{\prime})$ on $\mathbb{T}^{d-1}$
bounded by%
\begin{align*}
&  U(d-1)(2^{d-2}\Omega\tau^{-1})^{w+1-j_{d}}\sup_{w-j_{d}-(d-1)+2\leqslant
|\alpha|\leqslant w-j_{d}+1}\sup_{x^{\prime}\in S_{d-1}}\left\vert \left(
\frac{\partial}{\partial x^{\prime}}\right)  ^{\alpha}\frac{\partial^{j_{d}}%
g}{\partial x_{d}^{j_{d}}}(x^{\prime},0)\right\vert \\
&  \leqslant U(d-1)(2^{d-2}\Omega\tau^{-1})^{w+1-j_{d}}\sup_{w-d+3\leqslant
|\alpha|\leqslant w+1}\sup_{x\in S_{d}}\left\vert \frac{\partial^{\alpha}%
g}{\partial x^{\alpha}}(x)\right\vert .
\end{align*}
Hence, $-\tau^{-j_{d}-1}B_{j_{d}+1}\left(  x_{d}\right)  F\left(  x^{\prime
}\right)  $ is a function on $\mathbb{T}^{d}$ bounded by%
\begin{align*}
&  \tau^{-j_{d}-1}(2\pi)^{-j_{d}-1}\frac{\pi^{2}}{3}U(d-1)(2^{d-2}\Omega
\tau^{-1})^{w+1-j_{d}}\sup_{w-d+3\leqslant|\alpha|\leqslant w+1}\sup_{x\in
S_{d}}\left\vert \frac{\partial^{\alpha}g}{\partial x^{\alpha}}(x)\right\vert
\\
&  =(2^{d-2}\Omega2\pi)^{-j_{d}-1}\frac{\pi^{2}}{3}U(d-1)(2^{d-2}\Omega
\tau^{-1})^{w+2}\sup_{w-d+3\leqslant|\alpha|\leqslant w+1}\sup_{x\in S_{d}%
}\left\vert \frac{\partial^{\alpha}g}{\partial x^{\alpha}}(x)\right\vert ,
\end{align*}
and with Fourier expansion%
\begin{align*}
&  \sum_{n_{d}\neq0,n^{\prime}\in\mathcal{Q}(\theta_{1})\cap\mathbb{Z}^{d-1}%
}\left[  (2\pi i\tau n_{d})^{-j_{d}-1}\mathcal{R}_{\theta_{1},w-j_{d}}\left(
\frac{\partial^{j_{d}}g}{\partial x_{d}^{j_{d}}}(\cdot,0),\tau n^{\prime
}\right)  \right]  e^{2\pi in\cdot x}\\
&  =\sum_{n\in\mathbb{Z}^{d}}\chi_{\left\{  n_{d}\neq0,n^{\prime}%
\in\mathcal{Q}(\theta_{1})\cap\mathbb{Z}^{d-1}\right\}  }(n)\left[  (2\pi
i\tau n_{d})^{-j_{d}-1}\mathcal{R}_{\theta_{1},w-j_{d}}\left(  \frac
{\partial^{j_{d}}g}{\partial x_{d}^{j_{d}}}(\cdot,0),\tau n^{\prime}\right)
\right]  e^{2\pi in\cdot x}.
\end{align*}
Since $\mathcal{Q}(\theta)\cap\mathbb{Z}^{d}\subset\left(  \mathcal{Q}%
(\theta_{1})\cap\mathbb{Z}^{d-1}\right)  \times\left(  \mathbb{Z\setminus
}\left\{  0\right\}  \right)  $, by Lemma \ref{Lemma Q(theta)},
\begin{align*}
&  \sum_{n\in\mathcal{Q}(\theta)\cap\mathbb{Z}^{d}}\chi_{\left\{  n_{d}%
\neq0,n^{\prime}\in\mathcal{Q}(\theta_{1})\cap\mathbb{Z}^{d-1}\right\}
}(n)\left[  (2\pi i\tau n_{d})^{-j-1}\mathcal{R}_{\theta_{1},w-j}\left(
\frac{\partial^{j}g}{\partial x_{d}^{j}}(\cdot,0),\tau n^{\prime}\right)
\right]  e^{2\pi in\cdot x}\\
&  =\sum_{n\in\mathcal{Q}\left(  \theta\right)  \cap\mathbb{Z}^{d}}\left[
\left(  2\pi i\tau n_{d}\right)  ^{-j-1}\mathcal{R}_{\theta_{1},w-j}\left(
\frac{\partial^{j}g}{\partial x_{d}^{j}}\left(  \cdot,0\right)  ,\tau
n^{\prime}\right)  \right]  e^{2\pi in\cdot x}%
\end{align*}
is the Fourier series of a function on $\mathbb{T}^{d}$ bounded by%
\[
c\left(  \theta\right)  (2^{d-2}\Omega2\pi)^{-j_{d}-1}\frac{\pi^{2}}%
{3}U(d-1)(2^{d-2}\Omega\tau^{-1})^{w+2}\sup_{w-d+3\leqslant|\alpha|\leqslant
w+1}\sup_{x\in S_{d}}\left\vert \frac{\partial^{\alpha}g}{\partial x^{\alpha}%
}(x)\right\vert .
\]
Adding up on $j_{d}$ we obtain that
\[
\left\{  \chi_{\mathcal{Q}(\theta)}(n)\sum_{j_{d}=0}^{w}(2\pi i\tau
n_{d})^{-j_{d}-1}\mathcal{R}_{\theta_{1},w-j_{d}}\left(  \frac{\partial
^{j_{d}}g}{\partial x_{d}^{j_{d}}}(\cdot,0),\tau n^{\prime}\right)  \right\}
\]
are the Fourier coefficients of a function on the $d$-dimensional torus
bounded by%
\begin{align*}
&  \sum_{j_{d}=0}^{w}c\left(  \theta\right)  (2^{d-2}\Omega2\pi)^{-j_{d}%
-1}\frac{\pi^{2}}{3}U(d-1)(2^{d-2}\Omega\tau^{-1})^{w+2}\sup_{w-d+3\leqslant
|\alpha|\leqslant w+1}\sup_{x\in S_{d}}\left\vert \frac{\partial^{\alpha}%
g}{\partial x^{\alpha}}(x)\right\vert \\
&  =c\left(  \theta\right)  \frac{\pi^{2}}{3}U(d-1)(2^{d-2}\Omega\tau
^{-1})^{w+2}\sum_{j_{d}=0}^{w}(2^{d-2}\Omega2\pi)^{-j_{d}-1}\sup
_{w-d+3\leqslant|\alpha|\leqslant w+1}\sup_{x\in S_{d}}\left\vert
\frac{\partial^{\alpha}g}{\partial x^{\alpha}}(x)\right\vert \\
&  \leqslant c\left(  \theta\right)  \frac{\pi^{2}}{3}\frac{U(d-1)}%
{2^{d-1}\Omega\pi-1}(2^{d-2}\Omega\tau^{-1})^{w+2}\sup_{w-d+3\leqslant
|\alpha|\leqslant w+1}\sup_{x\in S_{d}}\left\vert \frac{\partial^{\alpha}%
g}{\partial x^{\alpha}}(x)\right\vert \\
&  \leqslant\left(  c\left(  \theta\right)  \frac{\pi^{2}}{3}\frac
{U(d-1)}{2^{d-1}\Omega\pi-1}2^{d-2}\Omega\tau_{0}^{-1}\right)  (2^{d-1}%
\Omega\tau^{-1})^{w+1}\sup_{w-d+3\leqslant|\alpha|\leqslant w+1}\sup_{x\in
S_{d}}\left\vert \frac{\partial^{\alpha}g}{\partial x^{\alpha}}(x)\right\vert
.
\end{align*}
Here we used the assumption $\Omega>(2\pi)^{-1}$.

The term $II$ is similar to $I$, but to estimate the remainder we need $(iii)$
in Lemma \ref{Lemma ovvio}. Let $\theta_{2}=\left\{  b^{\prime}:(b^{\prime
},-b^{\prime}\cdot\mathbf{1}_{d-1})\in\theta\right\}  \in\Theta_{d-1}$ (see Lemma \ref{Lemma 22} for details). We
claim that%
\[
\xi=\left(  \xi^{\prime},\xi_{d}\right)  \in\mathcal{Q}(\theta)\implies
\xi^{\prime}-\xi_{d}\mathbf{1}_{d-1}\in\mathcal{Q}(\theta_{2}).
\]
Indeed, let $\xi\in\mathcal{Q}(\theta)$ and let $b^{\prime}\in\theta_{2}%
\cap\Big(  \bigcup\limits_{\mathcal{B}\in\mathcal{F}^{\left(  d-1\right)  }%
}\mathcal{B}\Big)  $. Then $(b^{\prime},-b^{\prime}\cdot\mathbf{1}_{d-1}%
)\in\theta$ and therefore%
\[
0=\xi\cdot(b^{\prime},-b^{\prime}\cdot\mathbf{1}_{d-1})=\xi^{\prime}\cdot
b^{\prime}-\xi_{d}\mathbf{1}_{d-1}\cdot b^{\prime}=(\xi^{\prime}-\xi
_{d}\mathbf{1}_{d-1})\cdot b^{\prime}.
\]
Similarly if $b^{\prime}\in\Big(  \bigcup\limits_{\mathcal{B}\in
\mathcal{F}^{\left(  d-1\right)  }}\mathcal{B}\Big)  \setminus\theta_{2}$,
then $\left(  b^{\prime},-b^{\prime}\cdot\mathbf{1}_{d-1}\right)  \notin
\theta$. Since $\left(  b^{\prime},-b^{\prime}\cdot\mathbf{1}_{d-1}\right)
\in\bigcup\limits_{\mathcal{B}\in\mathcal{F}^{\left(  d\right)  }}\mathcal{B}%
$, we have%
\[
0\neq\xi\cdot\left(  b^{\prime},-b^{\prime}\cdot\mathbf{1}_{d-1}\right)
=(\xi^{\prime}-\xi_{d}\mathbf{1}_{d-1})\cdot b^{\prime}.
\]
By applying the $\left(  d-1\right)  $-dimensional formula corresponding to
$\theta_{2}$ to the function $\frac{\partial^{j_{d}}g}{\partial x_{d}^{j_{d}}%
}(x^{\prime},1-x^{\prime}\cdot\mathbf{1}_{d-1})$ we get
\begin{align*}
&  II=-\sum_{j_{d}=0}^{w}(2\pi i\xi_{d})^{-j_{d}-1}\int_{S_{d-1}}e^{-2\pi
ix^{\prime}\cdot\xi^{\prime}}e^{-2\pi i(1-x^{\prime}\cdot\mathbf{1}_{d-1}%
)\xi_{d}}\frac{\partial^{j_{d}}g}{\partial x_{d}^{j_{d}}}(x^{\prime
},1-x^{\prime}\cdot\mathbf{1}_{d-1})dx^{\prime}\\
&  =-\sum_{j_{d}=0}^{w}(2\pi i\xi_{d})^{-j_{d}-1}e^{-2\pi i\xi_{d}}%
\int_{S_{d-1}}e^{-2\pi ix^{\prime}\cdot(\xi^{\prime}-\xi_{d}\mathbf{1}_{d-1}%
)}\frac{\partial^{j_{d}}g}{\partial x_{d}^{j_{d}}}(x^{\prime},1-x^{\prime
}\cdot\mathbf{1}_{d-1})dx^{\prime}\\
&  =-\sum_{j_{d}=0}^{w}(2\pi i\xi_{d})^{-j_{d}-1}e^{-2\pi i\xi_{d}}\\
&  \quad\times\!\!\sum_{V^{\prime}\in\{1,2\}^{d-1}}\!\sum_{|J^{\prime}|\leqslant
w-j_{d}}\!\!\!\!\frac{\Big\langle \alpha(\theta_{2},V^{\prime},J^{\prime}),\left(
\left(  \partial/\partial x_{d}\right)  ^{j_{d}}g\right)  (x^{\prime
},1-x^{\prime}\cdot\mathbf{1}_{d-1})\Big\rangle e^{-2\pi i\lambda
_{V^{\prime}}\cdot(\xi^{\prime}-\xi_{d}\mathbf{1}_{d-1})}}{\prod
_{b_{k}^{\prime}\in\mathcal{B}_{V^{\prime}}\setminus\theta_{2}}\left(  2\pi
ib_{k}^{\prime}\cdot(\xi^{\prime}-\xi_{d}\mathbf{1}_{d-1})\right)  ^{j_{k}+1}%
}\\
&  \quad-\sum_{j_{d}=0}^{w}(2\pi i\xi_{d})^{-j_{d}-1}e^{-2\pi i\xi_{d}%
}\mathcal{R}_{\theta_{2},w-j_{d}}\left(  \frac{\partial^{j_{d}}g}{\partial
x_{d}^{j_{d}}}(x^{\prime},1-x^{\prime}\cdot\mathbf{1}_{d-1}),\xi^{\prime}%
-\xi_{d}\mathbf{1}_{d-1}\right) \\
&  =\sum_{V=(V^{\prime},2)\in\{1,2\}^{d}}\sum_{|J|\leqslant w}\frac
{-\left\langle \alpha(\theta_{2},V^{\prime},J^{\prime}),\left(  \left(
\partial/\partial x_{d}\right)  ^{j_{d}}g\right)  (x^{\prime},1-x^{\prime
}\cdot\mathbf{1}_{d-1})\right\rangle e^{-2\pi i\lambda_{V}\cdot\xi}}%
{\prod_{b_{k}\in\mathcal{B}_{V}\setminus\theta}(2\pi ib_{k}\cdot\xi)^{j_{k}%
+1}}\\
&  \quad-\sum_{j_{d}=0}^{w}(2\pi i\xi_{d})^{-j_{d}-1}e^{-2\pi i\xi_{d}%
}\mathcal{R}_{\theta_{2},w-j_{d}}\left(  \frac{\partial^{j_{d}}g}{\partial
x_{d}^{j_{d}}}(x^{\prime},1-x^{\prime}\cdot\mathbf{1}_{d-1}),\xi^{\prime}%
-\xi_{d}\mathbf{1}_{d-1}\right)  .
\end{align*}
By induction%
\begin{align*}
\Big\vert\Big\langle\alpha(\theta_{2},V^{\prime},J^{\prime}),  &  \left(
\left(  \partial/\partial x_{d}\right)  ^{j_{d}}g\right)  (x^{\prime
},1-x^{\prime}\cdot\mathbf{1}_{d-1})\Big\rangle\Big\vert\\
&  \leqslant c2^{(d-2)|J^{\prime}|}\sup_{|\alpha|\leqslant|J^{\prime}|}%
\sup_{x^{\prime}\in S_{d-1}}\left\vert \left(  \frac{\partial}{\partial
x^{\prime}}\right)  ^{\alpha}\left(  \frac{\partial^{j_{d}}g}{\partial
x_{d}^{j_{d}}}(x^{\prime},1-x^{\prime}\cdot\mathbf{1}_{d-1})\right)
\right\vert .
\end{align*}
Observe that for every $1\leqslant k\leqslant d-1$%
\[
\frac{\partial}{\partial x_{k}}\left(  \frac{\partial^{j_{d}}g}{\partial
x_{d}^{j_{d}}}(x^{\prime},1-x^{\prime}\cdot\mathbf{1}_{d-1})\right)
=\frac{\partial^{j_{d}+1}g}{\partial x_{k}\partial x_{d}^{j_{d}}}(x^{\prime
},1-x^{\prime}\cdot\mathbf{1}_{d-1})-\frac{\partial^{j_{d}+1}g}{\partial
x_{d}^{j_{d}+1}}(x^{\prime},1-x^{\prime}\cdot\mathbf{1}_{d-1}).
\]
Hence%
\[
\left\vert \left(  \frac{\partial}{\partial x^{\prime}}\right)  ^{\alpha
}\left(  \frac{\partial^{j_{d}}g}{\partial x_{d}^{j_{d}}}(x^{\prime
},1-x^{\prime}\cdot\mathbf{1}_{d-1})\right)  \right\vert \leqslant2^{|\alpha
|}\sup_{|\beta|=|\alpha|+j_{d}}\sup_{x\in S_{d}}\left\vert \frac
{\partial^{\beta}g}{\partial x^{\beta}}(x)\right\vert ,
\]
so that, with the notation $J=(J^{\prime},j_{d})$,%
\begin{align*}
&  \left\vert \left\langle \alpha(\theta_{2},V^{\prime},J^{\prime}),\left(
\left(  \partial/\partial x_{d}\right)  ^{j_{d}}g\right)  (x^{\prime
},1-x^{\prime}\cdot\mathbf{1}_{d-1})\right\rangle \right\vert \\
&  \leqslant c2^{(d-2)|J^{\prime}|}2^{|J^{\prime}|}\sup_{|\beta|\leqslant
|J^{\prime}|+j_{d}}\sup_{x\in S_{d}}\left\vert \frac{\partial^{\beta}%
g}{\partial x^{\beta}}(x)\right\vert \leqslant c2^{(d-1)|J|}\sup
_{|\beta|\leqslant|J|}\sup_{x\in S_{d}}\left\vert \frac{\partial^{\beta}%
g}{\partial x^{\beta}}(x)\right\vert .
\end{align*}
Let us consider the remainder and set $T\xi=(\xi^{\prime}-\xi_{d}%
\mathbf{1}_{d-1}\mathbf{,}\xi_{d})$. Then $T\in SL(d,\mathbb{Z})$ and, by $(iii)$   
in Lemma \ref{Lemma ovvio},
\begin{equation}
-\sum_{j_{d}=0}^{w}\chi_{\mathcal{Q}(\theta)}(n)(2\pi i\tau n_{d})^{-j_{d}%
-1}e^{-2\pi i\tau n_{d}}\mathcal{R}_{\theta_{2},w-j_{d}}\bigg(  \frac
{\partial^{j_{d}}g}{\partial x_{d}^{j_{d}}}(x^{\prime},1-x^{\prime}%
\cdot\mathbf{1}_{d-1}),\tau n^{\prime}-\tau n_{d}\mathbf{1}_{d-1}\bigg)
\label{label1}%
\end{equation}
are the Fourier coefficients of a bounded function if%
\begin{equation}
-\sum_{j_{d}=0}^{w}\chi_{T\left(  \mathcal{Q}(\theta)\right)  }(n)(2\pi i\tau
n_{d})^{-j_{d}-1}e^{-2\pi i\tau n_{d}}\mathcal{R}_{\theta_{2},w-j_{d}}\bigg(
\frac{\partial^{j_{d}}g}{\partial x_{d}^{j_{d}}}(x^{\prime},1-x^{\prime}%
\cdot\mathbf{1}_{d-1}),\tau n^{\prime}\bigg)  \label{label2}%
\end{equation}
are the Fourier coefficients of a bounded function, and the bound is the same.
By induction%
\[
\left\{  \chi_{\mathcal{Q}(\theta_{2})}(n^{\prime})\mathcal{R}_{\theta
_{2},w-j_{d}}\left(  \frac{\partial^{j_{d}}g}{\partial x_{d}^{j_{d}}%
}(x^{\prime},1-x^{\prime}\cdot\mathbf{1}_{d-1}),\tau n^{\prime}\right)
\right\}
\]
are the Fourier coefficients of a function on the $d-1$ dimensional torus
bounded by%
\begin{align*}
&  U(d-1)(2^{d-2}\Omega\tau^{-1})^{w+1-j_{d}}\sup_{\substack{w-j_{d}-(d-1)+2\leqslant
|\alpha|\\\leqslant w-j_{d}+1}}\sup_{x^{\prime}\in S_{d-1}}\left\vert \left(
\frac{\partial}{\partial x^{\prime}}\right)  ^{\alpha}\left(  \frac
{\partial^{j_{d}}g}{\partial x_{d}^{j_{d}}}(x^{\prime},1-x^{\prime}%
\cdot\mathbf{1}_{d-1})\right)  \right\vert \\
&  \leqslant U(d-1)(2^{d-2}\Omega\tau^{-1})^{w+1-j_{d}}\sup_{w-j_{d}%
-(d-1)+2\leqslant|\alpha|\leqslant w-j_{d}+1}2^{|\alpha|}\sup_{|\beta
|=|\alpha|+j_{d}}\sup_{x\in S_{d}}\left\vert \frac{\partial^{\beta}g}{\partial
x^{\beta}}(x)\right\vert \\
&  \leqslant U(d-1)(2^{d-1}\Omega\tau^{-1})^{w+1-j_{d}}\sup_{w-j_{d}%
-(d-1)+2\leqslant|\alpha|\leqslant w-j_{d}+1}\sup_{|\beta|=|\alpha|+j_{d}}%
\sup_{x\in S_{d}}\left\vert \frac{\partial^{\beta}g}{\partial x^{\beta}%
}(x)\right\vert \\
&  \leqslant U(d-1)(2^{d-1}\Omega\tau^{-1})^{w+1-j_{d}}\sup_{w-d+3\leqslant
|\beta|\leqslant w+1}\sup_{x\in S_{d}}\left\vert \frac{\partial^{\beta}%
g}{\partial x^{\beta}}(x)\right\vert .
\end{align*}
Hence, by $(i)$ in Lemma \ref{Lemma ovvio} and Corollary \ref{Lemma Bernoulli},%
\[
-\!\!\!\sum_{n_{d}\neq0,n^{\prime}\in\mathcal{Q}(\theta_{2})}\left[  \sum_{j_{d}%
=0}^{w}(2\pi i\tau n_{d})^{-j_{d}-1}e^{-2\pi i\tau n_{d}}\mathcal{R}%
_{\theta_{2},w-j_{d}}\left(  \frac{\partial^{j_{d}}g}{\partial x_{d}^{j_{d}}%
}(x^{\prime},1-x^{\prime}\cdot\mathbf{1}_{d-1}),\tau n^{\prime}\right)
\right]  e^{2\pi in\cdot x}%
\]
is the Fourier series of a function on $\mathbb{T}^{d}$ bounded by%
\begin{align*}
&  \sum_{j_{d}=0}^{w}\left(  \frac{\pi^{2}}{3}(2\pi\tau)^{-j_{d}%
-1}U(d-1)(2^{d-1}\Omega\tau^{-1})^{w+1-j_{d}}\sup_{w-d+3\leqslant
|\beta|\leqslant w+1}\sup_{x\in S_{d}}\left\vert \frac{\partial^{\beta}%
g}{\partial x^{\beta}}(x)\right\vert \right) \\
&  \leqslant\frac{\pi^{2}}{3}U(d-1)(2^{d-1}\Omega\tau^{-1})^{w+2}\sum
_{j_{d}=0}^{w}(2^{d-1}2\pi\Omega)^{-j_{d}-1}\sup_{w-d+3\leqslant
|\beta|\leqslant w+1}\sup_{x\in S_{d}}\left\vert \frac{\partial^{\beta}%
g}{\partial x^{\beta}}(x)\right\vert \\
&  \leqslant\frac{\pi^{2}}{3}U(d-1)(2^{d-1}\Omega\tau^{-1})^{w+2}\frac
{1}{2^{d-1}2\pi\Omega-1}\sup_{w-d+3\leqslant|\beta|\leqslant w+1}\sup_{x\in
S_{d}}\left\vert \frac{\partial^{\beta}g}{\partial x^{\beta}}(x)\right\vert .
\end{align*}
Since $T\mathcal{Q}(\theta)\cap\mathbb{Z}^{d}\subset\left(  \mathcal{Q}%
(\theta_{2})\cap\mathbb{Z}^{d-1}\right)  \times\left(  \mathbb{Z\setminus
}\left\{  0\right\}  \right)  $ (notice that $\xi\in\mathcal{Q}(\theta)$
implies $\xi^{\prime}-\xi_{d}\mathbf{1}_{d-1}$ $\in\mathcal{Q}(\theta_{2})$),
by Lemma \ref{Lemma Q(theta)}, the remainder (\ref{label2}) and therefore
(\ref{label1}) is bounded by%
\begin{align*}
&  c\frac{\pi^{2}}{3}\frac{U(d-1)}{2^{d-1}2\pi\Omega-1}(2^{d-1}\Omega\tau
^{-1})^{w+2}\sup_{w-d+3\leqslant|\beta|\leqslant w+1}\sup_{x\in S_{d}%
}\left\vert \frac{\partial^{\beta}g}{\partial x^{\beta}}(x)\right\vert \\
&  \leqslant c\left(  \frac{\pi^{2}}{3}\frac{U(d-1)}{2^{d-1}2\pi\Omega
-1}(2^{d-1}\Omega\tau_{0}^{-1})\right)  (2^{d-1}\Omega\tau^{-1})^{w+1}%
\sup_{w-d+3\leqslant|\beta|\leqslant w+1}\sup_{x\in S_{d}}\left\vert
\frac{\partial^{\beta}g}{\partial x^{\beta}}(x)\right\vert .
\end{align*}

This proves the formula when $e_{d}\notin\theta$. In particular,
\begin{align*}
&  \left\langle \alpha\left(  \theta,V,J\right)  ,g\right\rangle   =\left\{
\begin{array}
[c]{ll}%
\left\langle \alpha\left(  \theta_{1},V^{\prime},J^{\prime}\right)  ,\left(
\frac{\partial}{\partial x_{d}}\right)  ^{j_{d}}g\left(  x^{\prime},0\right)
\right\rangle  & V=\left(  V^{\prime},1\right)  ,\\
-\left\langle \alpha\left(  \theta_{2},V^{\prime},J^{\prime}\right)  ,\left(
\frac{\partial}{\partial x_{d}}\right)  ^{j_{d}}g\left(  x^{\prime
},1-x^{\prime}\cdot\boldsymbol{1}_{d-1}\right)  \right\rangle  & V=\left(
V^{\prime},2\right)  .
\end{array}
\right.
\end{align*}

\end{proof}
As mentioned, our formulas are not symmetric since they depend both on the
faces of $S_{d}$ that are orthogonal to the considered point $\xi$ and on the
way we iterate the integration in our computation. However one can obtain more
symmetric formulas by averaging on all different ways of computing the Fourier
transform $\widehat{g\chi_{S_{d}}}\left(  \xi\right)  $ as an iterated
integral.\medskip

In the following we obtain a precise expression for the functionals $\alpha(
\theta,V,J) $ that makes explicit the dependence of the coefficients
$\left\langle \alpha( \theta,V,J) ,g\right\rangle $ on the function $g$ and on
$\theta$, $V$ and $J$.

\begin{definition}
For every $h=2,\ldots,d$, let $U_{h}:\mathbb{R}^{h}\rightarrow\mathbb{R}%
^{h-1}$ be the operator that removes the last coordinate. For every $\theta
\in\Theta_{h}$ and for every $\mathcal{B}\in\mathcal{F}^{\left(  h\right)  }$,
let $P_{h,\mathcal{B}}\theta$ be the subspace of $\mathbb{R}^{h-1}$ defined as
follows:%
\[
P_{h,\mathcal{B}}\theta=%
\begin{cases}
U_{h}\theta & \text{if }e_{h}\in\theta,\\
U_{h}\left(  \theta\cap\{x\cdot e_{h}=0\}\right)  & \text{if }e_{h}%
\notin\theta\text{ and }\mathcal{B}\in\mathcal{F}_{1}^{(h)},\\
U_{h}\left(  \theta\cap\{x\cdot\mathbf{1}_{h}=0\}\right)  & \text{if }%
e_{h}\notin\theta\text{ and }\mathcal{B}\in\mathcal{F}_{2}^{(h)}.
\end{cases}
\]

\end{definition}

Observe that $P_{d,\mathcal{B}}\theta$ is just the subspace $\theta^{\prime}$,
$\theta_{1}$, or $\theta_{2}$ used in the proof of Lemma
\ref{Lemma trasf simplesso standard}.

\begin{lemma}
\label{Lemma 22}With the above notation we have%
\[
P_{h,\mathcal{B}}\theta=%
\begin{cases}
U_{h}\left(  \theta\cap\{x\cdot e_{h}=0\}\right)  & \text{if }\mathcal{B}%
\in\mathcal{F}_{1}^{(h)},\\
U_{h}\left(  \theta\cap\{x\cdot\mathbf{1}_{h}=0\}\right)  & \text{if
}\mathcal{B}\in\mathcal{F}_{2}^{(h)}.
\end{cases}
\]
Moreover, $P_{h,\mathcal{B}}\theta\in\Theta_{h-1}$.
\end{lemma}

\begin{proof}
Let $\Pi_{h}:\mathbb{R}^{h}\rightarrow\mathbb{R}^{h}$ be the orthogonal
projection
\[
\Pi_{h}(x_{1},\ldots,x_{h})=(x_{1},\ldots,x_{h-1},0)
\]
so that $U_{h}=U_{h}\Pi_{h}$. It suffices to observe that if $e_{h}\in\theta,$
then%
\[
\Pi_{h}\theta=\Pi_{h}\left(  \theta\cap\{x\cdot e_{h}=0\}\right)  =\Pi
_{h}\left(  \theta\cap\{x\cdot\mathbf{1}_h=0\}\right)  .
\]
Indeed, if $x\in\Pi_{h}\theta$ then $x=\Pi_{h}y$ for some $y\in\theta$, so
that $x=\Pi_{h}(\Pi_{h}y)$. Since $\Pi_{h}y\in\theta\cap\{x\cdot e_{h}=0\}$ we
have $\Pi_{h}\theta\subseteq\Pi_{h}\left(  \theta\cap\{x\cdot e_{h}%
=0\}\right)  $. Similarly $x=\Pi_{h}(\Pi_{h}y-(y\cdot\boldsymbol{1}_h)e_{h})$
and $\Pi_{h}y-(y\cdot\boldsymbol{1}_h)e_{h}\in\theta\cap\{x\cdot\boldsymbol{1}_h%
=0\}$ so that $\Pi_{h}\theta\subseteq\Pi_{h}(\theta\cap\{x\cdot\mathbf{1}_h%
=0\})$. The reverse inclusions are trivial. 

Concerning the last point of the lemma, observe first that it follows easily
from the recurrence definition of the bases of $\mathcal{F}^{(h)}$ that%
\[
\bigcup_{\mathcal{B}\in\mathcal{F}^{(h)}}\mathcal{B=}\left\{  e_{j}%
,e_{\ell}-e_{k}:1\leq j\leq h, 1\leq \ell<
k\leq h\right\}  .
\]
Assume now that $\theta=\operatorname*{span}\left\{  b_{1},\ldots
,b_{N}\right\}  $ with $b_{j}\in\bigcup_{\mathcal{B}\in\mathcal{F}^{(h)}%
}\mathcal{B}$. If $e_{h}\in\theta$, then $P_{h,\mathcal{B}}\theta=U_{h}%
\theta=\operatorname*{span}\left\{  U_{h}b_{1},\ldots,U_{h}b_{N}\right\}  $,
and by construction each vector $U_{h}b_{j}$ belongs to $\bigcup
_{\mathcal{B}\in\mathcal{F}^{(h-1)}}\mathcal{B}$, so that $P_{h,\mathcal{B}%
}\theta\in\Theta_{h-1}$. Assume $\mathcal{B\in F}_{1}^{(h)}$ (the case
$\mathcal{B\in F}_{2}^{(h)}$ is treated similarly). If $e_{h}\notin\theta$ and
$b_{j}\cdot e_{h}=0$ for all $j=1,\ldots,N$, then again $P_{h,\mathcal{B}%
}\theta=U_{h}\theta=\operatorname*{span}\left\{  U_{h}b_{1},\ldots,U_{h}%
b_{N}\right\}  \in\Theta_{h-1}$. If instead $e_{h}\notin\theta$ but, say,
$b_{N}\cdot e_{h}\neq0$, then in particular $b_{N}\cdot e_{h}=-1$ and
\begin{align*}
P_{h,\mathcal{B}}\theta & =U_{h}\left(  \theta\cap\left\{  x\cdot
e_{h}=0\right\}  \right)  =U_{h}\left\{  \sum_{j=1}^{N}c_{j}b_{j}:\sum
_{j=1}^{N}c_{j}b_{j}\cdot e_{h}=0\right\}  \\
& =U_{h}\left\{  \sum_{j=1}^{N}c_{j}b_{j}:c_{N}=\sum_{j=1}^{N-1}c_{j}%
b_{j}\cdot e_{h}\right\}  =\left\{  \sum_{j=1}^{N-1}c_{j}U_{h}\left(
b_{j}+\left(  b_{j}\cdot e_{h}\right)  b_{N}\right)  \right\}  .
\end{align*}
Now if $b_{j}\cdot e_{h}=0$ then $U_{h}\left(  b_{j}+\left(  b_{j}\cdot
e_{h}\right)  b_{N}\right)  =U_{h}b_{j}\in\bigcup_{\mathcal{B}\in
\mathcal{F}^{(h-1)}}\mathcal{B}$. If on the contrary $b_{j}\cdot e_{h}\neq0$,
then $b_{j}\cdot e_{h}=-1$ and either $ b_{j}+\left(  b_{j}\cdot
e_{h}\right)  b_{N}  =b_{j}-b_{N}$ or its opposite belong to
$\bigcup_{\mathcal{B}\in\mathcal{F}^{(h)}}\mathcal{B}$. It follows that
$U_{h}\left(  b_{j}+\left(  b_{j}\cdot e_{h}\right)  b_{N}\right)  $ or
$-U_{h}\left(  b_{j}+\left(  b_{j}\cdot e_{h}\right)  b_{N}\right)  $ belong
to $\bigcup_{\mathcal{B}\in\mathcal{F}^{(h-1)}}\mathcal{B}$. Thus,
$P_{h,\mathcal{B}}\theta\in\Theta_{h-1}$.

\end{proof}

\begin{definition}
For every $h=2,\ldots,d$, and for every $\mathcal{B}=\{b_{1},\ldots,b_{h}%
\}\in\mathcal{F}^{\left(  h\right)  }$, define $P_{h}\mathcal{B}\in
\mathcal{F}^{\left(  h-1\right)  }$ as $P_{h}\mathcal{B}=\{U_{h}b_{1}%
,\ldots,U_{h}b_{h-1}\}$. Observe that by Definition \ref{bases-defn},
$P_{h}\mathcal{B}$ is the basis used to construct $\mathcal{B}$.
\end{definition}

\textbf{Notation}. In the next lemmas and definitions for a given basis
$\mathcal{B\in F}^{\left(  d\right)  }$, we shall call $\mathcal{B}%
_{d}=\mathcal{B}$, $\mathcal{B}_{d-1}=P_{d}\mathcal{B}$, $\mathcal{B}%
_{d-2}=P_{d-1}P_{d}\mathcal{B}$, \ldots, $\mathcal{B}_{2}=P_{3}P_{4}\ldots
P_{d}\mathcal{B}$, and $\mathcal{B}_{1}=P_{2}P_{3}\ldots P_{d}\mathcal{B}%
=\{1\}.$ Also, we denote by $b_{k}^{(h)}$ the $k$-th vector of the basis
$\mathcal{B}_{h}$. Similarly, $e_{k}^{(h)}$ denotes the $k$-th vector of the
canonical basis of $\mathbb{R}^{h}$.

\begin{lemma}
\label{Lemma Step}For all $h\geqslant2$ and for all $k=1,\ldots,h-1$,
$b_{k}^{(h-1)}=U_{h}b_{k}^{(h)}\in P_{h,\mathcal{B}}\theta$ if and only if
$b_{k}^{(h)}\in\theta.$
\end{lemma}

\begin{proof}
If $\mathcal{B\in F}_{1}^{(h)}$ then $b_{k}^{(h)}\cdot e_{h}^{(h)}=0$. Thus if
$b_{k}^{(h)}\in\theta$ then, by Lemma \ref{Lemma 22}, $U_{h}b_{k}^{(h)}\in
P_{h,\mathcal{B}}\theta$. Conversely, if $U_{h}b_{k}^{(h)}\in P_{h,\mathcal{B}%
}\theta$ then there exists $y\in\theta\cap\{x\cdot e_{h}^{(h)}=0\}$ such that
$U_{h}b_{k}^{(h)}=U_{h}y$, so that $U_{h}(b_{k}^{(h)}-y)=0$. This implies that
the first $h-1$ coordinates of $b_{k}^{(h)}$ and $y$ coincide. The last
coordinate of $y$ is $0$ by construction, and since $\mathcal{B\in F}%
_{1}^{(h)}$ and $k<h$, the last coordinate of $b_{k}^{(h)}$ also is $0$. Thus
$b_{k}^{(h)}=y$ so that $b_{k}^{(h)}\in\theta$.

Similarly, if $\mathcal{B\in F}_{2}^{(h)}$ then $b_{k}^{(h)}\cdot\mathbf{1}_h%
=0$. Thus if $b_{k}^{(h)}\in\theta$ then $U_{h}b_{k}^{(h)}\in P_{h,\mathcal{B}%
}\theta$. Conversely, if $U_{h}b_{k}^{(h)}\in P_{h,\mathcal{B}}\theta$ then
there exists $y\in\theta\cap\{x\cdot\mathbf{1}_h=0\}$ such that $U_{h}%
b_{k}^{(h)}=U_{h}y$, so that $U_{h}(b_{k}^{(h)}-y)=0$. This implies that the
first $h-1$ coordinates of $b_{k}^{(h)}$ and $y$ coincide. The fact that
$b_{k}^{(h)}\cdot\mathbf{1}_h=0$ and $y\cdot\mathbf{1}_h=0$ implies that also the last coordinate of
$b_{k}^{(h)}$ and $y$ coincide. Thus $b_{k}^{(h)}=y$ so that $b_{k}^{(h)}%
\in\theta$.
\end{proof}

\begin{definition}
For any given $\theta\in\Theta_{d}$ and for every $\mathcal{B}\in
\mathcal{F}^{\left(  d\right)  }$ the multi-index $Z=Z_{\mathcal{B},\theta
}=(z_{1},\ldots,z_{d})\in\{0,1\}^{d}$ is defined recursively as follows:%
\begin{align*}
z_{d}  &  =0\text{ iff }e_{d}^{(d)}\in\theta\\
z_{d-1}  &  =0\text{ iff }e_{d-1}^{(d-1)}\in P_{d,\mathcal{B}_{d}}\theta\\
z_{d-2}  &  =0\text{ iff }e_{d-2}^{(d-2)}\in P_{d-1,\mathcal{B}_{d-1}%
}P_{d,\mathcal{B}_{d}}\theta\\
z_{d-3}  &  =0\text{ iff }e_{d-3}^{(d-3)}\in P_{d-2,\mathcal{B}_{d-2}%
}P_{d-1,\mathcal{B}_{d}}P_{d,\mathcal{B}_{d}}\theta\\
&  \ldots\\
z_{1}  &  =0\text{ iff }1\in P_{2,\mathcal{B}_{2}}\ldots P_{d-2,\mathcal{B}%
_{d-2}}P_{d-1,\mathcal{B}_{d-1}}P_{d,\mathcal{B}_{d}}\theta
\end{align*}

\end{definition}

\begin{definition}
For any given $\theta\in\Theta_{d}$ and any multi-index $V\in\left\{
1,2\right\}  ^{d}$ define the multi-index $I=I_{V,\theta}=(i_{1},\ldots
,i_{d})\in\{0,1\}^{d}$ by
\[
i_{k}=0\iff b_{k}^{(d)}\in\theta
\]
where $\mathcal{B}_{V}=\left\{  b_{1}^{(d)},\ldots,b_{d}^{(d)}\right\}
\in\mathcal{F}^{\left(  d\right)  }$ is the basis associated with $V$.
\end{definition}

\begin{lemma}
For any given $\theta\in\Theta_{d}$ and for every $\mathcal{B}_{V}%
\in\mathcal{F}^{\left(  d\right)  }$, $Z_{\mathcal{B}_{V},\theta}=I_{V,\theta
}$.
\end{lemma}

\begin{proof}
By definition, it suffices to observe that $z_{h}=0$ if and only if
$e_{h}^{(h)}$ belongs to $P_{h+1,\mathcal{B}_{h+1}}\left(  \ldots
P_{d-1,\mathcal{B}_{d-1}}(P_{d,\mathcal{B}_{d}}\theta)\right)  $. But
$e_{h}^{(h)}=b_{h}^{(h)}$ is the last vector of the basis $\mathcal{B}_{h}$ so
that, by Lemma \ref{Lemma Step}, $e_{h}^{(h)}=b_{h}^{(h)}\in
P_{h+1,\mathcal{B}_{h+1}}\left(  \ldots P_{d-1,\mathcal{B}_{d-1}%
}(P_{d,\mathcal{B}_{d}}\theta)\right)  $ $\iff$ $b_{h}^{(h+1)}\in
P_{h+2,\mathcal{B}_{h+2}}\left(  \ldots\left(  P_{d-1,\mathcal{B}_{d-1}%
}(P_{d,\mathcal{B}_{d}}\theta)\right)  \right)  $ $\iff\dots\iff$
$b_{h}^{(d-1)}\in P_{d,\mathcal{B}_{d}}\theta$ $\iff$ $b_{h}^{(d)}\in\theta$
$\iff$ $i_{h}=0$.
\end{proof}

\begin{definition}
\label{Def mu}Fix $V=(v_{1},v_{2},\ldots,v_{d})\in\{1,2\}^{d}$, $I=(i_{1}%
,\ldots,i_{d})\in\{0,1\}^{d}$, $J=(j_{1},\ldots,j_{d})$ a non-negative
multi-index such that $j_{h}=0$ if $i_{h}=0$. For each $h=1,\ldots,d$ and for
$N\geq1$, define the operators%
\[
T_{h}^{v_{h},i_{h},j_{h}}:\mathcal{C}^{N}(\mathbb{R}^{h})\rightarrow
\mathcal{C}^{N-1}(\mathbb{R}^{h-1})
\]
(if $h=1$ then $\mathcal{C}^{N-1}(\mathbb{R}^{h-1})=\mathbb{C}$) as follows:
if $h=1$, set
\begin{align*}
T_{1}^{1,0,0}g  &  =\int_{0}^{1}g(x_{1})dx_{1},\\
T_{1}^{2,0,0}g  &  =0,\\
T_{1}^{1,1,j_{1}}g  &  =-\frac{d^{j_{1}}g}{dx^{j_{1}}}(0),\\
T_{1}^{2,1,j_{1}}g  &  =\frac{d^{j_{1}}g}{dx^{j_{1}}}(1).
\end{align*}
If $2\leq h\leq d$, for all $x^{\prime}\in\mathbb{R}^{h-1}$, set
\begin{align*}
T_{h}^{1,0,0}g(x^{\prime})  &  =\int_{0}^{1-x^{\prime}\cdot\mathbf{1}%
}g(x^{\prime},x_{h})dx_{h},\\
T_{h}^{2,0,0}g(x^{\prime})  &  =0,\\
T_{h}^{1,1,j_{h}}g(x^{\prime})  &  =-\frac{\partial^{j_{h}}g}{\partial
x_{h}^{j_{h}}}(x^{\prime},0),\\
T_{h}^{2,1,j_{h}}g(x^{\prime})  &  =\frac{\partial^{j_{h}}g}{\partial
x_{h}^{j_{h}}}(x^{\prime},1-x^{\prime}\cdot\mathbf{1}).
\end{align*}

Let us define the integro-differential functionals%
\[
\mu(V,I,J) =T_{1}^{v_{1},i_{1},j_{1}}T_{2}^{v_{2},i_{2},j_{2}}\ldots
T_{d}^{v_{d},i_{d},j_{d}}.
\]

\end{definition}

\begin{lemma}
Fix $V=(v_{1},v_{2},\ldots,v_{d})\in\{1,2\}^{d}$, $I=(i_{1},\ldots,i_{d}%
)\in\{0,1\}^{d}$ and let $J=(j_{1},\ldots,j_{d})$ be a non-negative
multi-index such that $j_{h}=0$ if $i_{h}=0$. Let $\mathcal{B=B}_{V}%
\in\mathcal{F}^{\left(  d\right)  }$ be the basis associated to the
multi-index V, and let $\theta\in\Theta_{d}$ be such that
$I_{V,\mathcal{\theta}}=I$. Then%
\[
\alpha(\theta,V,J)=(-1)^{|I|}\mu(V,I,J).
\]

\end{lemma}

\begin{proof}
One has to go through the proof of Lemma
\ref{Lemma trasf simplesso standard} and notice that $\theta^{\prime}$,
$\theta_{1}$ and $\theta_{2}$ are all simply $P_{d,\mathcal{B}_{d}}\theta$,
and that $\mathcal{B}^{\prime}=P_{d}\mathcal{B=B}_{d-1}$. The conclusion
follows proceeding recursively and recalling that $Z_{\mathcal{B},\theta
}=I_{V,\theta}$.
\end{proof}

The above functional $\alpha(\theta,V,J)$ is a compactly supported
distribution, with support contained in the simplex $S_{d}$. In particular,
the dependence of $\alpha(\theta,V,J)$ on $V$, $\theta$ and $J$ is condensed
in the multi-indices $V$ and $I$. Recall that each $V\in\{1,2\}^{d}$
determines a unique basis in $\mathcal{F}^{\left(  d\right)  }$, precisely
$\mathcal{B}_{V}$. On the other hand, given a basis $\mathcal{B}\in
\mathcal{F}^{\left(  d\right)  }$ associated with the multi-index $V$, for any
$I\in\{0,1\}^{d}$ there might be several subspaces $\theta\in\Theta_{d}$ such
that $I_{V,\theta}=I$. By the above lemma, all these subspaces therefore
produce identical coefficients $\alpha(\theta,V,J)$.

Notice that if $v_{1}=2$, then $\mu(V,I,J)$ reduces to a linear combination of
derivatives of the Dirac delta centered at $( 1,0,\ldots,0) $ of order at most
$\vert J\vert$. This follows easily from the fact that the only point in the
simplex $S_{d}$ with first coordinate equal to $1$ is $( 1,0,\ldots,0) $. If
furthermore $i_{1}=0,$ then $\mu(V,I,J) =0$.

Assume $v_{1}=1$. We already mentioned that the support of $\mu(V,I,J)$ is
contained in the simplex $S_{d}$. Furthermore, for any $h\geq2$,

\begin{enumerate}
\item if $( v_{h},i_{h}) =( 1,1) $ then the support of $\mu(V,I,J)$ is
contained in the hyperplane $x_{h}=0.$

\item If $( v_{h},i_{h}) =( 2,1) $ then the support of $\mu(V,I,J)$ is
contained in the hyperplane $x_{h}=1-( x_{1}+\ldots+x_{h-1}) .$

\item If $( v_{h},i_{h}) =( 2,0) $ then $\mu(V,I,J)=0$.

\item The couple $( v_{h},i_{h}) =( 1,0) $ gives no restrictions on the
support of $\mu(V,I,J)$.
\end{enumerate}

Similarly, in the case $h=1$,

\begin{enumerate}\setcounter{enumi}{4}
\item if $( v_{1},i_{1}) =( 1,1) $ then the support of $\mu(V,I,J)$ is
contained in the hyperplane $x_{1}=0.$

\item The couple $( v_{1},i_{1}) =( 1,0) $ gives no restrictions on the
support of $\mu(V,I,J)$.
\end{enumerate}

\begin{remark}
If $g$ is smooth with compact support in $S_{d}$, then $\widehat{g\chi_{Sd}%
}(\xi)=\widehat{g}(\xi)$ has fast decay at infinity. Observe that this does
not contradict the above theorem. Indeed, by the previous remarks all
coefficients $\left\langle \mu(V,I,J),g\right\rangle $ vanish except when
$V=(1,\ldots,1)$ and $I=(0,\ldots,0)$ which implies that $J=(0,\ldots,0)$.
This choice of $V$ and $I$ forces $\theta=\mathbb{R}^{d}$ and $\mathcal{Q}%
(\theta)=\{0\}$. In this case we have%
\[
\widehat{g\chi_{Sd}}(0)=\int_{S_{d}}g(x)dx.
\]
For $\xi\neq0$ all the coefficients $\left\langle \mu(V,I,J),g\right\rangle $
vanish so that%
\[
\widehat{g\chi_{Sd}}(\xi)=\mathcal{R}_{\theta,w}(g,\xi).
\]

\end{remark}

\subsection{General Simplex\label{section-general}}

With an affine change of variables Lemma \ref{Lemma trasf simplesso standard}
for the standard simplex can be transferred to a general simplex.

\begin{definition}
Let $M\in GL(d,\mathbb{Z})$ and let $\mathcal{B}=\{b_{1},\ldots,b_{d}%
\}\in\mathcal{F}^{\left(  d\right)  }$. Then we shall denote by $M\mathcal{B}$
the basis $\{Mb_{1},\ldots,Mb_{d}\}$ and by $M\mathcal{F}^{\left(  d\right)
}$ the collection of the bases $M\mathcal{B}$ with $\mathcal{B}\in
\mathcal{F}^{\left(  d\right)  }$. Similarly $M\Theta_{d}$ is the collection
of all the spaces $M\theta$ with $\theta\in\Theta_{d}$. Clearly $M\Theta_{d}$
consists of all subspaces generated by any possible choice of vectors in
$M\mathcal{F}^{\left(  d\right)  }$. For every $\eta\in M\Theta_{d}$ we set%
\[
\mathcal{Q}_{M}(\eta)=\left\{  \xi\in\mathbb{R}^{d}:\text{for all }v\in
\bigcup\limits_{\mathcal{B}\in\mathcal{F}^{\left(  d\right)  }}M\mathcal{B}%
\text{, }\xi\cdot v=0~\text{iff }v\in\eta\right\}  .
\]

\end{definition}

\begin{lemma}
Let $M\in GL( d,\mathbb{Z}) $. For every $\theta\in\Theta_{d}$%
\[
\mathcal{Q}_{M}(M\theta)=( M^{t}) ^{-1}\mathcal{Q}( \theta).
\]

\end{lemma}

\begin{proof}
This follows immediately from the definitions. Observe that $\xi\in
(M^{t})^{-1}\mathcal{Q}(\theta)$ if and only if $M^{t}\xi\in\mathcal{Q}%
(\theta)$ if and only if, for all $b\in\bigcup\limits_{\mathcal{B}%
\in\mathcal{F}^{\left(  d\right)  }}\mathcal{B}$,
\[
M^{t}\xi\cdot b=\xi\cdot Mb=0~\text{iff }b\in\theta,
\]
if and only if, for every $v\in\bigcup\limits_{\mathcal{B}\in\mathcal{F}%
^{\left(  d\right)  },}M\mathcal{B}$,
\[
\xi\cdot v=0\text{ iff }v\in M\theta\text{,}%
\]
if and only if $\xi\in\mathcal{Q}_{M}(M\theta)$.
\end{proof}

\begin{lemma}
\label{Lemma trasf simplesso generico}Let $\mathcal{P}$ be a simplex in
$\mathbb{R}^{d}$ with vertices $\mathbf{0},\mathbf{m}_{1},\ldots
,\mathbf{m}_{d}\in\mathbb{Z}^{d}$ and let $M\in GL(d,\mathbb{Z})$ be the
$d\times d$ matrix with columns $\mathbf{m}_{1},\mathbf{m}_{2},\ldots
,\mathbf{m}_{d}$, which maps the standard simplex $S_{d}$ onto $\mathcal{P}$.
Let $q\in C^{w+1}(\mathbb{R}^{d})$ with $w\in\mathbb{N}$, let $Q(x)=q(x)\chi
_{\tau\mathcal{P}}(x)$ with $\tau>0$, and let $q_{\tau,M}(x)=q(\tau Mx)$.
Then, following the definitions and notations of the previous section, for
every $\theta\in\Theta_{d}$ and $\xi\in\mathcal{Q}_{M}(M\theta)$,%
\begin{align*}
\widehat{Q}(\xi)  &  =\int_{\tau\mathcal{P}}q(x)e^{-2\pi ix\cdot\xi}dx\\
&  =\tau^{d}\det(M)\sum_{V\in\{1,2\}^{d}}\sum_{|J|\leqslant w,J\sqsubseteq
I_{V,\theta}}\frac{(-1)^{|I_{V,\theta}|}\left\langle \mu(V,I_{V,\theta
},J),q_{\tau,M}\right\rangle e^{-2\pi i\tau M\lambda_{V}\cdot\xi}}%
{\prod_{b_{k}\in\mathcal{B}_{V}\setminus\theta}\left(  2\pi i\tau Mb_{k}%
\cdot\xi\right)  ^{j_{k}+1}}\\
&  \quad+\tau^{d}\det(M)\mathcal{R}_{\theta,w}(q_{\tau,M},\tau M^{t}\xi).
\end{align*}
In the above formula we adopt the convention: $\mathcal{B}_{V}=\{b_{1}%
,\ldots,b_{d}\}$ is the basis associated with the multi-index $V=(v_{1}%
,v_{2},\ldots,v_{d})$, $I_{V,\theta}=(i_{1},\ldots,i_{d})\in\{0,1\}^{d}$ is
the multi-index such that $i_{k}=0$ if and only if $b_{k}\in\theta$, $J\sqsubseteq I_{V,\theta}$ means that $j_k=0$ if $i_k=0$. The
coefficients $\left\langle \mu(V,I_{V,\theta},J),q_{\tau,M}\right\rangle $ and
the remainder $\mathcal{R}_{\theta,w}(q_{\tau,M},\tau M^{t}\xi)$ are the ones
defined in Lemma \ref{Lemma trasf simplesso standard} and Definition
\ref{Def mu}. In particular they satisfy the following:

(i) the coefficients $\left\langle \mu(V,I_{V,\theta},J),q_{\tau,
M}\right\rangle $ satisfy the estimate%
\[
\left\vert \left\langle \mu(V,I_{V,\theta},J),q_{\tau, M}\right\rangle
\right\vert \leqslant c 2^{( d-1) \vert J\vert}\sup_{\vert\alpha\vert
\leqslant\vert J\vert}\sup_{x\in S_{d}}\left\vert \frac{\partial^{\alpha
}q_{\tau, M}}{\partial x^{\alpha}}( x) \right\vert .
\]

(ii) The remainder $\mathcal{R}_{\theta,w}( q_{\tau, M},\tau M^{t}\xi) $ has
the following property: for every $\Omega>1/( 2\pi) $ and every $\tau_{0}>0$
there exists a constant $c=c( d,\Omega,\tau_{0}) >0$ independent of $q_{\tau,
M}$ and $w$ such that for every $\tau>\tau_{0}$ the coefficients $\left\{
\chi_{\mathcal{Q}_{M}(M\theta)}( n) \mathcal{R}_{\theta,w}( q_{\tau, M},\tau
M^{t}n) \right\}  _{n\in\mathbb{Z}^{d}}$ are the Fourier coefficients of a
function on the torus $\mathbb{T}^{d}$ bounded by%
\[
c ( 2^{d-1}\Omega\tau^{-1}) ^{w+1}\sup_{w-d+2\leqslant\vert\alpha
\vert\leqslant w+1}\sup_{x\in S_{d}}\left\vert \frac{\partial^{\alpha}q_{\tau,
M}}{\partial x^{\alpha}}( x) \right\vert .
\]

\end{lemma}

\begin{proof}
This lemma follows from Lemma \ref{Lemma trasf simplesso standard} via an
affine change of variables. Define%
\[
G( x) =Q( \tau Mx) =q( \tau Mx) \chi_{\tau\mathcal{P}}( \tau Mx) =q_{\tau, M}(
x) \chi_{S_{d}}( x) .
\]
Then,%
\begin{align*}
\widehat{Q}( \xi)  &  =\int_{\mathbb{R}^{d}}Q( x) e^{-2\pi i\xi\cdot x}%
dx=\tau^{d}\det( M) \int_{\mathbb{R}^{d}}Q( \tau Mx) e^{-2\pi i\xi\cdot\tau
Mx}dx\\
&  =\tau^{d}\det( M) \int_{\mathbb{R}^{d}}Q( \tau Mx) e^{-2\pi i\tau M^{t}%
\xi\cdot x}dx\\
&  =\tau^{d}\det( M) \widehat{G}( \tau M^{t}\xi) .
\end{align*}
Applying Lemma \ref{Lemma trasf simplesso standard} to the function $G( x) $
we obtain the desired expansion. The same lemma also shows that $\left\{
\chi_{\mathcal{Q}( \theta) }( n) \mathcal{R}_{\theta,w}( q_{\tau,M},\tau n)
\right\}  _{n\in\mathbb{Z}^{d}}$ are the Fourier coefficients of a function on
the torus bounded by%
\[
U ( 2^{d-1}\Omega\tau^{-1}) ^{w+1}\sup_{w-d+2\leqslant\vert\alpha
\vert\leqslant w+1}\sup_{x\in S_{d}}\left\vert \frac{\partial^{\alpha}q_{\tau,
M}}{\partial x^{\alpha}}( x) \right\vert
\]
where $U$ is the same constant that appears in Lemma
\ref{Lemma trasf simplesso standard}. By Lemma \ref{annichilatore} in Appendix
A, $\left\{  \mathcal{\chi}_{\mathcal{Q}(\theta)}( M^{t}n) \mathcal{R}%
_{\theta,w}( q_{\tau, M},\tau M^{t}n) \right\}  _{n\in\mathbb{Z}^{d}}$ are the
Fourier coefficients of a function on the torus satisfying the same bound.
\end{proof}

\section{Expansion in multivariate Bernoulli polynomials\label{Sect main res}}

In this section we shall prove our Theorem \ref{main-thm}. Let us
start with a lemma on the Fourier expansion of the multivariate Bernoulli polynomials.

\begin{lemma}
\label{Lemma Fourier multi Bernoulli}Let $J=(j_{1},j_{2},...,j_{d})$ be a
multi-index of non-negative integers and let $L\in GL(d,\mathbb{Z})$. If
$\mathfrak{B}_{J,L}(x)$ are as in Definition \ref{def:Bernoulli poly}, then,
for every $x\in\mathbb{R}^{d}$,
\[
\mathfrak{B}_{J,L}(x)=\lim_{\varepsilon\rightarrow0+}\left\{  (-1)^{|I|}%
\sum_{n\in\Delta(I,L)}\widehat{\varphi}(\varepsilon n)\dfrac{e^{2\pi in\cdot
x}}{(2\pi iLn)^{J}}\right\}  .
\]
Here, $I=(i_{1},\ldots,i_{d})$ with $i_{k}=0$ if $j_{k}=0$ and $i_{k}=1$ if
$j_{k}>0$, that is, $J\sqsubseteq I$, and the set $\Delta(I,L)$ is the subset of frequencies in
$\mathbb{Z}^{d}$ defined by%
\[
\Delta(I,L)=\left\{  n\in\mathbb{Z}^{d}:\ (Ln)_{k}=0\ \text{iff}%
\ i_{k}=0\right\}  .
\]
Finally, in the denominators
\[
(2\pi iLn)^{J}=(2\pi i\left(  Ln\right)  _{1})^{j_{1}}(2\pi i\left(
Ln\right)  _{2})^{j_{2}}\cdots(2\pi i\left(  Ln\right)  _{d})^{j_{d}}%
\]
we adopt the convention that $0^{0}=1$. In particular, all the denominators in
the Fourier expansion of $\mathfrak{B}_{J,L}(x)$ are different from zero.
\end{lemma}

\begin{proof}
Recall that if $f(x)$ is an integrable function with Fourier transform
$\widehat{f}(\xi)$ and $L$ is a non-singular matrix, then $\widehat{f}(L\xi)$
is the Fourier transform of $|\det L|^{-1}f\left(  (L^{-1})^{t}x\right)  $.
Moreover, if $f(x)$ is a function with bounded support, the Poisson summation
formula gives
\[
|\det L|^{-1}\sum_{n\in\mathbb{Z}^{d}}\varphi_{\varepsilon}\ast f\left(
(L^{-1})^{t}(x+n)\right)  =\sum_{n\in\mathbb{Z}^{d}}\widehat{\varphi
}(\varepsilon n)\widehat{f}(Ln)e^{2\pi inx}.
\]
Observe that the series on the left is finite and the one on the right is
absolutely convergent, so that the application of the summation formula is
legitimate. Then the lemma follows by choosing $f(x)=B_{J}(x)=B_{j_{1}}%
(x_{1})\cdots B_{j_{d}}(x_{d})$. Indeed, for every $n$ in $\mathbb{Z}^{d}$,
one has
\begin{align*}
&  \int_{\mathbb{R}^{d}}B_{J}(x)e^{-2\pi in\cdot x}dx=\prod_{k=1}^{d}\int
_{0}^{1}B_{j_{k}}(x_{k})e^{-2\pi in_{k}\cdot x_{k}}dx_{k}\\
&  =\prod_{k=1}^{d}%
\begin{cases}
-1/\left(  2\pi in_{k}\right)  ^{j_{k}} & \text{if }j_{k}\neq0\text{ and
}n_{k}\neq0,\\
0 & \text{if }j_{k}\neq0\text{ and }n_{k}=0,\\
0 & \text{if }j_{k}=0\text{ and }n_{k}\neq0,\\
1 & \text{if }j_{k}=0\text{ and }n_{k}=0.
\end{cases}
\end{align*}
Hence, by the definition of $\Delta(I,L)$,
\[
\int_{\mathbb{R}^{d}}B_{J}(x)e^{-2\pi iLn\cdot x}dx=%
\begin{cases}
(-1)^{|I|}/(2\pi iLn)^{J} & \text{if }n\in\Delta(I,L),\\
0 & \text{if }n\notin\Delta(I,L).\
\end{cases}
\]

\end{proof}

We shall also need the following lemma.

\begin{lemma}
\label{Lemma unione}For a fixed $V\in\{ 1,2\} ^{d}$ and for every $I\in\{
0,1\} ^{d}$ we have%
\[
\bigcup_{\theta\in\Theta_{d}:I_{V,\theta}=I}\left[  \mathbb{Z}^{d}%
\cap\mathcal{Q}_{M}( M\theta) \right]  =\Delta\left(  I,( MD_{V}) ^{t}\right)
.
\]
Here if, as usual, $\mathcal{B}_{V}=\{b_{1},\ldots,b_{d}\}$ is the basis
associated with the multi-index $V$, then $I_{V,\theta}=( i_{1},\ldots,i_{d})
$ where $i_{k}=0$ if and only if $b_{k}\in\theta$, and $D_{V}$ is the matrix
with columns $b_{1},\ldots,b_{d}$.
\end{lemma}

\begin{proof}
Assume that
\[
m\in\bigcup_{\theta\in\Theta_{d}:I_{V,\theta}=I}\left[  \mathbb{Z}^{d}%
\cap\mathcal{Q}_{M}( M\theta) \right]
\]
Then $m\in\mathcal{Q}_{M}( M\theta) $ for some $\theta$ such that
$I_{V,\theta}=I$. Thus, if $b_{k}\in\mathcal{B}_{V}$, then $m\cdot Mb_{k}=0$
if and only if $b_{k}\in\theta$, but since $I_{V,\theta}=I$, then $b_{k}%
\in\theta$ if and only if $i_{k}=0$. Thus $m\cdot Mb_{k}=0$ if and only if
$i_{k}=0$, which implies $m\in\Delta\left(  I,( MD_{V}) ^{t}\right)  $, since
\begin{align*}
\Delta\left(  I,( MD_{V}) ^{t}\right)   &  =\left\{  m\in\mathbb{Z}^{d}: \ (
MD_{V}) ^{t}m\cdot e_{k}=0\text{ iff }i_{k}=0\right\} \\
&  =\left\{  m\in\mathbb{Z}^{d}\text{: }\left(  ( MD_{V}) ^{t}m\right)
^{t}e_{k}=0\text{ iff }i_{k}=0\right\} \\
&  =\left\{  m\in\mathbb{Z}^{d}\text{: }m^{t}MD_{V}e_{k}=0\text{ iff }%
i_{k}=0\right\} \\
&  =\left\{  m\in\mathbb{Z}^{d}\text{: }m^{t}Mb_{k}=0\text{ iff }%
i_{k}=0\right\} \\
&  =\left\{  m\in\mathbb{Z}^{d}\text{: }m\cdot Mb_{k}=0\text{ iff }%
i_{k}=0\right\}  .
\end{align*}

Conversely, if $m\in\Delta\left(  I,( MD_{V}) ^{t}\right)  ,$ that is if
$m\in\mathbb{Z}^{d}$ is such that $Mb_{k}\cdot m=0$ if and only if $i_{k}=0,$
then, calling%
\[
\theta_{m}=\left\langle b\text{ in the bases}:Mb\cdot m=0\right\rangle ,
\]
we have $I_{V,\theta_{m}}=I$, (indeed, setting $I_{V,\theta_{m}}=(
r_{1},\ldots,r_{d}) $, we have $r_{k}=0$ if and only if $b_{k}\in\theta_{m}$
if and only if $Mb_{k}\cdot m=0$ if and only if $i_{k}=0$). Finally,
obviously, $m\in Q_{M}( M\theta_{m}) $.
\end{proof}
We are ready to prove our main result.
\begin{proof}
[Proof of Theorem \ref{main-thm}]Let $Q(x)=\chi_{\tau\mathcal{P}}(x)q(x)$. We
have%
\begin{align*}
\sum_{n\in\mathbb{Z}^{d}}\varphi_{\varepsilon}\ast Q(x+n)  &  =\sum
_{n\in\mathbb{Z}^{d}}\widehat{\varphi}(\varepsilon n)\widehat{Q}(n)e^{2\pi
in\cdot x}\\
&  =\sum_{\theta\in\Theta_{d}}\sum_{n\in\mathbb{Z}^{d}\cap\mathcal{Q}%
_{M}(M\theta)}\widehat{\varphi}(\varepsilon n)\widehat{Q}(n)e^{2\pi in\cdot
x}.
\end{align*}
Since the multiplier $\widehat{\varphi}(\varepsilon n)$ is
rapidly decreasing the series converge absolutely and the rearrangements of
the terms of the series are allowed. Then, by Lemma
\ref{Lemma trasf simplesso generico},
\begin{align*}
&  \sum_{n\in\mathbb{Z}^{d}}\varphi_{\varepsilon}\ast Q(x+n)\\
& =\!\!\sum_{\theta\in\Theta_{d}}\!\!\sum_{\substack{n\in\mathbb{Z}^{d}\cap\\\mathcal{Q}%
_{M}(M\theta)}}\!\!\!\!\widehat{\varphi}(\varepsilon n)\tau^{d}\det(M)\!\!\!\sum
_{V\in\{1,2\}^{d}}\!\sum_{|J|\leqslant w}\!\!\!\frac{(-1)^{|I_{V,\theta}|}\left\langle
\mu(V,I_{V,\theta},J),q_{\tau,M}\right\rangle e^{-2\pi i(\tau M\lambda
_{V})\cdot n}e^{2\pi in\cdot x}}{\prod_{b_{k}\in\mathcal{B}_{V}\setminus\theta}(2\pi i\tau
Mb_{k}\cdot n)^{j_{k}+1}}\\
&  \quad+\sum_{\theta\in\Theta_{d}}\sum_{n\in\mathbb{Z}^{d}\cap\mathcal{Q}%
_{M}(M\theta)}\widehat{\varphi}(\varepsilon n)\tau^{d}\det(M)\mathcal{R}%
_{\theta,w}(q_{\tau,M},\tau M^{t}n)e^{2\pi in\cdot x}\\
&  =\sum_{V\in\{1,2\}^{d}}\sum_{\theta\in\Theta_{d}}\Phi(V,\theta
)+\sum_{\theta\in\Theta_{d}}\Psi(\theta)
\end{align*}
where%
\begin{align*}
\Phi(V,\theta)=\tau^{d}\det(M)\sum_{|J|\leqslant w}  &  \left\langle
\mu(V,I_{V,\theta},J),q_{\tau,M}\right\rangle \\
&  \times(-1)^{|I_{V,\theta}|}\sum_{n\in\mathbb{Z}^{d}\cap\mathcal{Q}%
_{M}(M\theta)}\widehat{\varphi}(\varepsilon n)\frac{\,e^{2\pi i(x-\tau
M\lambda_{V})\cdot n}}{\prod_{b_{k}\in\mathcal{B}_{V}\setminus\theta}(2\pi
i\tau Mb_{k}\cdot n)^{j_{k}+1}}.
\end{align*}
Rearranging the sum we have%
\begin{align*}
\sum_{V\in\{1,2\}^{d}}\sum_{\theta\in\Theta_{d}}&\Phi(V,\theta)   =\sum
_{V\in\{1,2\}^{d}}\sum_{I\in\{0,1\}^{d}}\sum_{\theta\in\Theta_{d}:I_{V,\theta
}=I}\Phi(V,\theta)\\
&  =\sum_{V\in\{1,2\}^{d}}\sum_{I\in\{0,1\}^{d}}\sum_{\theta\in\Theta
_{d}:I_{V,\theta}=I}\tau^{d}\det(M)\sum_{|J|\leqslant w}\left\langle
\mu(V,I,J),q_{\tau,M}\right\rangle \\
&  \quad\quad\times(-1)^{|I|}\sum_{n\in\mathbb{Z}^{d}\cap\mathcal{Q}%
_{M}(M\theta)}\widehat{\varphi}(\varepsilon n)\frac{e^{2\pi i(x-\tau
M\lambda_{V})\cdot n}}{\prod_{b_{k}\in\mathcal{B}_{V}\setminus\theta}(2\pi
i\tau Mb_{k}\cdot n)^{j_{k}+1}}\\
&  =\det(M)\sum_{V\in\{1,2\}^{d}}\sum_{I\in\{0,1\}^{d}}\sum_{|J|\leqslant
w}\tau^{d}\left\langle \mu(V,I,J),q_{\tau,M}\right\rangle \\
&  \quad\times(-1)^{|I|}\sum_{\theta\in\Theta_{d}:I_{V,\theta}=I}\sum
_{n\in\mathbb{Z}^{d}\cap\mathcal{Q}_{M}(M\theta)}\widehat{\varphi}(\varepsilon
n)\frac{e^{2\pi i(x-\tau M\lambda_{V})\cdot n}}{\prod_{b_{k}\in\mathcal{B}%
_{V}\setminus\theta}(2\pi i\tau Mb_{k}\cdot n)^{j_{k}+1}}\\
&  =\det(M)\sum_{V\in\{1,2\}^{d}}\sum_{I\in\{0,1\}^{d}}\sum_{|J|\leqslant
w}\tau^{d-|J|-|I|}\left\langle \mu(V,I,J),q_{\tau,M}\right\rangle \\
&  \quad\times(-1)^{|I|}\sum_{\theta\in\Theta_{d}:I_{V,\theta}=I}\sum
_{n\in\mathbb{Z}^{d}\cap\mathcal{Q}_{M}(M\theta)}\widehat{\varphi}(\varepsilon
n)\frac{e^{2\pi i(x-\tau M\lambda_{V})\cdot n}}{\prod_{b_{k}\in\mathcal{B}%
_{V}}(2\pi iMb_{k}\cdot n)^{j_{k}+i_{k}}}%
\end{align*}
with the usual convention that in the denominators $0^{0}=1$. Now, since
$\mathbb{Z}^{d}\cap\left(  \bigcup_{\theta:I_{V,\theta}=I}\mathcal{Q}%
_{M}(M\theta)\right)  =\Delta\left(  I,(MD_{V})^{t}\right)  $ by Lemma
\ref{Lemma unione}, we have
\begin{align*}
\sum_{\theta\in\Theta_{d}:I_{V,\theta}=I}\sum_{n\in\mathbb{Z}^{d}%
\cap\mathcal{Q}_{M}(M\theta)}\widehat{\varphi}  &  (\varepsilon n)\frac
{e^{2\pi i(x-\tau M\lambda_{V})\cdot n}}{\prod_{b_{k}\in\mathcal{B}_{V}}(2\pi
iMb_{k}\cdot n)^{j_{k}+i_{k}}}\\
&  =\sum_{n\in\Delta\left(  I,(MD_{V})^{t}\right)  }\widehat{\varphi
}(\varepsilon n)\frac{e^{2\pi in\cdot(x-\tau M\lambda_{V})}}{\prod_{b_{k}%
\in\mathcal{B}_{V}}(2\pi iMb_{k}\cdot n)^{j_{k}+i_{k}}}.
\end{align*}
Observe that $Mb_{k}\cdot n=(b_{k}^{t}M^{t})n=(D_{V}^{t}M^{t})_{k}n=\left(
(MD_{V})^{t}n\right)  _{k}=\left(  (MD_{V})^{t}n\right)  \cdot e_{k}$. Hence,
\begin{align*}
&  (-1)^{|I|}\sum_{n\in\Delta\left(  I,(MD_{V})^{t}\right)  }\widehat{\varphi
}(\varepsilon n)\frac{e^{2\pi in\cdot(x-\tau M\lambda_{V})}}{\prod_{b_{k}%
\in\mathcal{B}_{V}}(2\pi iMb_{k}\cdot n)^{j_{k}+i_{k}}}\\
&  =(-1)^{|I|}\sum_{n\in\Delta\left(  I,(MD_{V})^{t}\right)  }\widehat
{\varphi}(\varepsilon n)\frac{e^{2\pi in\cdot(x-\tau M\lambda_{V})}}%
{\prod_{b_{k}\in\mathcal{B}_{V}}\left(  2\pi i\left(  (MD_{V})^{t}n\right)
\cdot e_{k}\right)  ^{j_{k}+i_{k}}}\\
&  =(-1)^{|I|}\sum_{n\in\Delta\left(  I,(MD_{V})^{t}\right)  }\widehat
{\varphi}(\varepsilon n)\frac{e^{2\pi in\cdot(x-\tau M\lambda_{V})}}{\left(
2\pi i(MD_{V})^{t}n\right)  ^{J+I}}\\
&  =\varphi_{\varepsilon}\ast\mathfrak{B}_{J+I,(MD_{V})^{t}}(x-\tau
M\lambda_{V}).
\end{align*}
Since $J\sqsubseteq I$, the vanishing
components of $J+I$ appear in the same spots as those of $I$, so that
$\Delta\left(  I,(MD_{V})^{t}\right)  =\Delta\left(  I+J,(MD_{V})^{t}\right)
$. Hence, the principal part becomes
\begin{align*}
&  \sum_{V\in\{1,2\}^{d}}\sum_{\theta\in\Theta_{d}}\Phi(V,\theta)\\
&  =\det(M)\!\!\!\sum_{V\in\{1,2\}^{d}}\sum_{I\in\{0,1\}^{d}}\sum_{|J|\leqslant
w}\tau^{d-|J|-|I|}\left\langle \mu(V,I,J),q_{\tau,M}\right\rangle
\varphi_{\varepsilon}\ast\mathfrak{B}_{J+I,(MD_{V})^{t}}(x-\tau M\lambda_{V}).
\end{align*}
Let us consider now the remainder%
\[
\Psi(\theta)=\tau^{d}\det(M)\sum_{n\in\mathbb{Z}^{d}\cap\mathcal{Q}%
_{M}(M\theta)}\widehat{\varphi}(\varepsilon n)\mathcal{R}_{\theta,w}%
(q_{\tau,M},\tau M^{t}n)e^{2\pi in\cdot x}.
\]
For every $\theta\in\Theta_{d}$, by Lemma \ref{Lemma trasf simplesso generico}%
, $\Psi(\theta)$ is a function bounded by%
\[
c\tau^{d-w-1}\det(M)(2^{d-2}\pi^{-1}+\delta)^{w+1}\sup_{w-d+2\leqslant
|\alpha|\leqslant w+1}\sup_{x\in S_{d}}\left\vert \frac{\partial^{\alpha
}q_{\tau,M}}{\partial x^{\alpha}}(x)\right\vert
\]
It follows that $\sum_{\theta\in\Theta_{d}}\Psi(\theta)$ is a bounded
function, with the same bound. Letting $\varepsilon\rightarrow0$ gives the
desired result.
\end{proof}

\section{Proofs of Theorems \ref{Thm 1} and \ref{Thm 2} \label{Sect proofs}}

Theorem \ref{Thm 1} and Theorem \ref{Thm 2} are corollaries of Theorem
\ref{main-thm}. In particular, Theorem \ref{Thm 1} follows applying the next result to a
decomposition of the given polytope into simplices.

\begin{theorem}
\label{Thm gammak simplex}Let $S_{d}$ be the standard simplex in
$\mathbb{R}^{d}$, let $\mathcal{P}=MS_{d}$ with $M\in GL(d,\mathbb{Z})$. Let
$p\in\mathbb{Z}^{d}$, let $w$ be a non-negative integer and let $f\in
C^{w+1}(\mathbb{R}^{d})$. Then, there exists a numerical sequence $\{\gamma
_{k}\}_{0<k\leqslant w/2}$ such that for every positive integer $N$ we have%
\[
N^{-d}\sum_{n\in\mathbb{Z}^{d}}\omega_{p+\mathcal{P}}(N^{-1}n)f(N^{-1}%
n)=\int_{p+\mathcal{P}}f(x)dx+\sum_{0<k\leqslant w/2}\gamma_{k}N^{-2k}%
+O(N^{-w-1}).
\]
More precisely, with the notation in Theorem \ref{main-thm},%
\[
\gamma_{k}=\det(M)\sum_{V\in\{1,2\}^{d}}\sum_{I\in\{0,1\}^{d}}\sum
_{J\sqsubseteq I,\ |I+J|=2k}\,\left\langle \mu(V,I,J),f(p+M\cdot)\right\rangle
\mathfrak{B}_{J+I,(MD_{V})^{t}}(0).
\]

\end{theorem}

\begin{proof}
Assume first $p=0$. Since%
\[
N^{-d}\sum_{n\in\mathbb{Z}^{d}}\omega_{\mathcal{P}}(N^{-1}n)f(N^{-1}%
n)=N^{-d}\sum_{n\in\mathbb{Z}^{d}}\omega_{N\mathcal{P}}(n)f(N^{-1}n),
\]
we apply Theorem \ref{EM-d} with $\tau=N$ to the function $q_{N}%
(x)=f(N^{-1}x)$. Let $g_{N,M}(x)=q_{N}(NMx)$, then%
\begin{align*}
&  N^{-d}\sum_{n\in\mathbb{Z}^{d}}\omega_{N\mathcal{P}}(n)f(N^{-1}n)\\
=  &  \det(M)\sum_{V\in\{1,2\}^{d}}\sum_{I\in\{0,1\}^{d}}\sum_{|J|\leqslant
w,J\sqsubseteq I}N^{-|J|-|I|}\left\langle \mu(V,I,J),g_{N,M}\right\rangle
\mathfrak{B}_{J+I,(MD_{V})^{t}}(0)\\
&  \quad+N^{-d}\mathcal{R}_{w}(0)
\end{align*}
with
\[
\left\vert N^{-d}\mathcal{R}_{w}(x)\right\vert \leqslant cN^{-w-1}%
\det(M)(2^{d-2}\pi^{-1}+\delta)^{w+1}\sup_{w-d+2\leqslant|\alpha|\leqslant
w+1}\sup_{x\in S_{d}}\left\vert \frac{\partial^{\alpha}g_{N,M}}{\partial
x^{\alpha}}(x)\right\vert.
\]
Since $g_{N,M}(x)=f\left(  Mx\right)  $, then%
\begin{align*}
&  N^{-d}\sum_{n\in\mathbb{Z}^{d}}\omega_{MS_{d}}(N^{-1}n)f(N^{-1}n)\\
=  &  \det(M)\sum_{V\in\{1,2\}^{d}}\sum_{I\in\{0,1\}^{d}}\sum_{|J|\leqslant
w,J\sqsubseteq I}N^{-|J|-|I|}\left\langle \mu(V,I,J),f\left(  M\cdot\right)
\right\rangle \mathfrak{B}_{J+I,(MD_{V})^{t}}(0)\\
&  +N^{-d}\mathcal{R}_{w}(0).
\end{align*}
Now observe that when $I=(0,\ldots,0)$ and therefore $J=(0,\ldots,0)$ we have
\[
\mu(V,I,J)=0
\]
if $V\neq(1,\ldots,1)$, whereas when $V=(1,\ldots,1)$ we have
\[
\left\langle \mu(V,I,J),f\left(  M\cdot\right)  \right\rangle =\int_{S_{d}%
}f(Mx)dx=\det(M)^{-1}\int_{MS_{d}}f(x)dx.
\]
Also observe that 
\[
\mathfrak{B}_{J+I,(MD_{V})^{t}}(0)=0
\]
whenever $|J|+|I|$ is odd. Indeed, since ${\Delta(I+J,(MD_{V})^{t})}$ is a cone and $\varphi$ is radial,
by Lemma \ref{Lemma Fourier multi Bernoulli} we have%
\begin{align*}
\mathfrak{B}_{I+J,(MD_{V})^{t}}(0)  &  =\lim_{\varepsilon\rightarrow
0+}\left\{  (-1)^{|I|}\sum_{n\in\Delta(I+J,(MD_{V})^{t})}\dfrac{\widehat
{\varphi}(\varepsilon n)}{\left(  2\pi i(MD_{V})^{t}n\right)  ^{I+J}}\right\}
\\
&  =\lim_{\varepsilon\rightarrow0+}\left\{  (-1)^{|I|}\sum_{n\in\Delta\left(
I+J,(MD_{V})^{t}\right)  }\dfrac{\widehat{\varphi}(\varepsilon n)}{(-2\pi
i(MD_{V})^{t}n)^{I+J}}\right\} \\
&  =(-1)^{|I+J|}\mathfrak{B}_{I+J,(MD_{V})^{t}}(0).
\end{align*}
Therefore,
\begin{align*}
&  N^{-d}\sum_{n\in\mathbb{Z}^{d}}f(N^{-1}n)\omega_{MS_{d}}(N^{-1}%
n)=\int_{MS_{d}}f(x)dx\\
&  +\sum_{k\geq1}N^{-2k}\Big(  \det(M)\sum_{V\in\{1,2\}^{d}}\sum
_{I\in\{0,1\}^{d}}\sum_{\substack{|J|\leqslant w,J\sqsubseteq I,\\\text{ }\left\vert
I+J\right\vert =2k}}\left\langle \mu(V,I,J),f\left(  M\cdot\right)
\right\rangle \mathfrak{B}_{J+I,(MD_{V})^{t}}(0)\Big) \\
&  +N^{-d}\mathcal{R}_{w}(0)\\
&  =\int_{MS_{d}}f(x)dx+\sum_{0<2k\leqslant w}\gamma_{k}N^{-2k}+O(N^{-w-1})
\end{align*}
where
\[
\gamma_{k}=\det(M)\sum_{V\in\{1,2\}^{d}}\sum_{I\in\{0,1\}^{d}}\sum
_{J\sqsubseteq I,\ |I+J|=2k}\left\langle \mu(V,I,J),f\left(  M\cdot\right)
\right\rangle \mathfrak{B}_{J+I,(MD_{V})^{t}}(0).
\]
Now assume $p\neq0$. Then%
\begin{align*}
&  N^{-d}\sum_{n\in\mathbb{Z}^{d}}\omega_{p+\mathcal{P}}(N^{-1}n)f(N^{-1}%
n)=N^{-d}\sum_{n\in\mathbb{Z}^{d}}\omega_{\mathcal{P}}(N^{-1}n-p)f(N^{-1}n)\\
&  =N^{-d}\sum_{n\in\mathbb{Z}^{d}}\omega_{\mathcal{P}}(N^{-1}n)f(N^{-1}n+p).
\end{align*}
Hence, the case $p\neq0$ follows from the case $p=0$ applied to the function
$f(x+p)$.
\end{proof}
Theorem \ref{Thm 1} and Theorem \ref{Thm 2} are now easily deduced.
\begin{proof}
[Proof of Theorem \ref{Thm 1}] We use Theorem
\ref{Thm gammak simplex} and the additivity of $S_{N}(f,\mathcal{P})$ with
respect to $\mathcal{P}$.
\end{proof}

\begin{proof}
[Proof of Theorem \ref{Thm 2}]By Theorem \ref{Thm 1},%
\[
S_{N}\left(  f,\mathcal{P}\right)  =\int_{\mathcal{P}}f\left(  x\right)
dx+\sum_{0<k\leqslant w/2}\gamma_{k}N^{-2k}+O\left(  N^{-w-1}\right)  .
\]

Then%
\begin{align*}
&  \sum_{0\leqslant j\leqslant w/2}c_{j}S_{2^{j}N}( f,\mathcal{P})\\
&  =\bigg(  \sum_{0\leqslant j\leqslant w/2}c_{j}\bigg)  \int_{\mathcal{P}}f(
x) dx+\sum_{0<k\leqslant w/2}\gamma_{k}N^{-2k}\bigg(  \sum_{0\leqslant
j\leqslant w/2}c_{j}2^{-2kj}\bigg)  +O( N^{-w-1})
\end{align*}
and the conclusion follows since the Vandermonde system is solvable.
\end{proof}

\section{Appendix A: Some basic facts on the harmonic analysis on commutative
groups}

\label{section-appendix}

The following results on the harmonic analysis on groups, subgroups and
quotient spaces are well known (see \cite[Section 2.7]{Rudin}). We include the
case of the torus for the reader's convenience.

\begin{definition}
Let $\mathcal{H}$ be a subgroup of $\mathbb{Z}^{d}$. The annihilator of
$\mathcal{H}$ is the compact subgroup $\mathcal{H}^{\perp}$ of $\mathbb{T}%
^{d}$ given by%
\[
\mathcal{H}^{\perp}=\left\{  t\in\mathbb{T}^{d}:\forall h\in\mathcal{H}%
,\ t\cdot h\in\mathbb{Z}\right\}  =\left\{  h\in\mathbb{Z}^{d}:\forall
t\in\mathcal{H}^{\perp}\text{, }e^{2\pi it\cdot h}=1\right\}  .
\]

\end{definition}

\begin{lemma}
\label{Dual anni}We have%
\[
\mathcal{H}=\left\{  h\in\mathbb{Z}^{d}:\forall t\in\mathcal{H}^{\perp}\text{,
}t\cdot h\in\mathbb{Z}\right\}  .
\]

\end{lemma}

This is a particular case of Lemma 2.1.3 in \cite{Rudin}. The following is a
direct elementary proof.

\begin{proof}
Let $h\in\mathcal{H}$, then by definition $t\cdot h\in\mathbb{Z}$ for every
$t\in\mathcal{H}^{\perp}$. To show the converse observe that since every
subgroup of $\mathbb{Z}^{d}$ is a lattice, there exists an integer $d\times q$
matrix $B$, with $q\leqslant d$, of maximal rank such that%
\[
\mathcal{H}=\left\{  Bz:z\in\mathbb{Z}^{q}\right\}  .
\]
Then%
\begin{align*}
\mathcal{H}^{\perp}  &  =\left\{  t\in\mathbb{T}^{d}:\forall z\in
\mathbb{Z}^{q}\text{, }Bz\cdot t\in\mathbb{Z}\right\}  =\left\{
t\in\mathbb{T}^{d}:\forall z\in\mathbb{Z}^{q}\text{, }z\cdot B^{t}%
t\in\mathbb{Z}\right\} \\
&  =\left\{  t\in\mathbb{T}^{d}:B^{t}t\in\mathbb{Z}^{q}\right\}  .
\end{align*}
Since $B$ has rank $q$ we can assume without loss of generality that there
exists a $q\times q$ invertible matrix $C$ and a $q\times\left(  d-q\right)  $
matrix $D$ such that $B^{t}=\left[
\begin{array}
[c]{cc}%
C & D
\end{array}
\right]  $. Hence for $t=\left(  t_{1},t_{2}\right)  $,
\[
B^{t}t=Ct_{1}+Dt_{2}=z\in\mathbb{Z}^{q}.
\]
It follows that%
\[
\mathcal{H}^{\perp}=\left\{  \left(  C^{-1}(z-Dt_{2}),t_{2}\right)
:z\in\mathbb{Z}^{q},t_{2}\in\mathbb{T}^{d-q}\right\}  .
\]
Now, let $m\in\mathbb{Z}^{d}$ such that for every $t\in\mathcal{H}^{\perp}$ we
have $m\cdot t\in\mathbb{Z}$. Then, if $m=\left(  m_{1},m_{2}\right)  $, for
every $z\in\mathbb{Z}^{q},t_{2}\in\mathbb{T}^{d-q}$ we have%
\begin{equation}
\left(  m_{1},m_{2}\right)  \cdot\left(  C^{-1}(z-Dt_{2}),t_{2}\right)
=m_{1}\cdot C^{-1}z-m_{1}\cdot C^{-1}Dt_{2}+t_{2}\cdot m_{2}\in\mathbb{Z}.
\label{1}%
\end{equation}
Let $t_{2}=0$. Then $m_{1}\cdot C^{-1}z\in\mathbb{Z}$ and hence $\left(
C^{-1}\right)  ^{t}m_{1}\cdot z\in\mathbb{Z}$ for every $z\in\mathbb{Z}^{q}$.
Therefore $\left(  C^{-1}\right)  ^{t}m_{1}\in\mathbb{Z}^{q}$. It follows that
$m_{1}=C^{t}h$ for some $h\in\mathbb{Z}^{q}$. From (\ref{1}) we obtain that
\begin{align*}
&  C^{t}h\cdot C^{-1}z-C^{t}h\cdot C^{-1}Dt_{2}+t_{2}\cdot m_{2}\\
&  =h\cdot z-h\cdot Dt_{2}+t_{2}\cdot m_{2}%
\end{align*}
is an integer for every $t_{2}\in\mathbb{T}^{d-q}$. It follows that for every
$t_{2}\in\mathbb{T}^{d-q}$ we have%
\[
(m_{2}-D^{t}h)\cdot t_{2}\in\mathbb{Z},
\]
this implies that $m_{2}=D^{t}h$ and therefore $m=Bh\in\mathcal{H}$.
\end{proof}

Let $d\mu$ be the Haar measure on $\mathcal{H}^{\perp}$. Since $\mathcal{H}%
^{\perp}$ is compact we can assume that $\left\vert d\mu\right\vert =1$.

\begin{lemma}
\label{Lemma int Hperp}With the normalized Haar measure $d\mu$ on
$\mathcal{H}^{\perp}$, for every $m\in\mathbb{Z}^{d}$ we have%
\[
\int_{\mathcal{H}^{\perp}}e^{2\pi im\cdot t}d\mu( t) =
\begin{cases}
1 & m\in\mathcal{H},\\
0 & m\not \in \mathcal{H}.
\end{cases}
\]

\end{lemma}

\begin{proof}
The case $m\in\mathcal{H}$ is immediate since $e^{2\pi im\cdot t}=1$ for every
$t\in\mathcal{H}^{\perp}$. Let $m\not \in \mathcal{H}$. By Lemma
\ref{Dual anni} there exists $t_{0}\in\mathcal{H}^{\perp}$ such that $e^{2\pi
im\cdot t_{0}}\neq1$. By the invariance of the Haar measure we have%
\[
\int_{\mathcal{H}^{\perp}}e^{2\pi im\cdot t}d\mu(t)=\int_{\mathcal{H}^{\perp}%
}e^{2\pi im\cdot(t_{0}+t)}d\mu(t)=e^{2\pi im\cdot t_{0}}\int_{\mathcal{H}%
^{\perp}}e^{2\pi im\cdot t}d\mu(t).
\]
Hence
\[
\int_{\mathcal{H}^{\perp}}e^{2\pi im\cdot t}d\mu(t)=0.
\]

\end{proof}

\begin{lemma}
\label{annichilatore}Let $\mathcal{H}$ be a subgroup of $\mathbb{Z}^{d}$ and
let $d\mu$ be the normalized Haar measure on the annihilator $\mathcal{H}%
^{\perp}$. In particular $d\mu$ is a probability measure on $\mathbb{T}^{d}$.
Let $f\in L^{1}(\mathbb{T}^{d})$ and let $g(s)=\mu\ast f(s)$, that is%
\[
g(s)=\int_{\mathcal{H}^{\perp}}f(s-t)d\mu(t).
\]
Then,

(i) $\left\Vert g\right\Vert _{\infty}\leqslant\left\Vert f\right\Vert
_{\infty}$;

(ii)%
\[
\widehat{g}(m)=%
\begin{cases}
\widehat{f}(m) & m\in\mathcal{H},\\
0 & m\not \in \mathcal{H}.
\end{cases}
\]

\end{lemma}

\begin{proof}
$(i)$ follows from the fact that the convolution with a probability measure is
an operator with norm $1$ on $L^{\infty}\left(  \mathbb{T}^{d}\right)  $.
$(ii)$ follows from the fact that $\widehat{\mu\ast f}(m)=\widehat{\mu
}(m)\widehat{f}(m)$ and%
\[
\widehat{\mu}(m)=\int_{\mathcal{H}^{\perp}}e^{-2\pi im\cdot t}d\mu(t)=%
\begin{cases}
1 & m\in\mathcal{H},\\
0 & m\not \in \mathcal{H}.
\end{cases}
\]

\end{proof}

\section{Appendix B: Bernoulli polynomials and Lerch Zeta functions}

\label{section-B}

Here we give a different description of the functions $\mathfrak{B}_{J,L}(x)$.
Such functions were defined (Definition \ref{def:Bernoulli poly}) starting
with a product of Bernoulli polynomials, restricting this product to the unit
cube, composing it with an affine transformation and finally periodizing. One
may ask if these operations commute and if these functions can be obtained as
a linear combination of affine transformation of the periodic multivariate
Bernoulli polynomials $\mathfrak{B}_{J,\text{Id}}$ (here Id denotes the
identity matrix).

We start from the Fourier expansion%
\[
\mathfrak{B}_{J,L}(x)=\lim_{\varepsilon\rightarrow0+}(-1)^{|I|}\sum
_{n\in\Delta(I,L)}\widehat{\varphi}(\varepsilon n)\frac{e^{2\pi inx}}{(2\pi
iLn)^{J}}%
\]
with the usual conventions on the notation (Lemma
\ref{Lemma Fourier multi Bernoulli}). In particular, the multi-index $I=(i_{1},\ldots,i_{d})$ is such that $i_{k}=0$ if $j_{k}=0$ and $i_{k}=1$ if
$j_{k}>0$. Recall that in the points of
discontinuity the definition of $\mathfrak{B}_{J,L}$ is by regularization and
that $L\in GL(d,\mathbb{Z})$. Assume now that $x$ is a point of continuity, so
that the mollifier $\varphi$ may not be taken to be necessarily radial. More
precisely, we may set $\varphi(x)=|\det L|^{-1}\psi\left(  (L^{T}%
)^{-1}x\right)  \ $where%
\[
\psi(x)=\eta(x_{1})\ldots\eta(x_{d})
\]
and $\eta$ is a non-negative smooth function with compact support and integral
one. In particular
\[
\widehat{\varphi}(\xi)=\widehat{\psi}(L\xi).
\]
Since $L$ has integer entries, $L$ has a unique (column) Hermite normal form
$H$, that is, $L=HU,$ where $H$ is a lower triangular matrix with positive
coefficients on the diagonal and such that all the other coefficients are
nonnegative and smaller than the diagonal coefficient in the same row, whereas
$U$ is a unimodular integer matrix. The invertibility of the linear map
$n\mapsto Un$ in $\mathbb{Z}^{d}$ immediately implies that the lattice
$L\mathbb{Z}^{d}$ coincides with the lattice $H\mathbb{Z}^{d}$, and more
specifically%
\[
\left\{  Ln:n\in\Delta(I,L)\right\}  =\left\{  HUn:n\in\Delta(I,L)\right\}
=\left\{  Hm:m\in\Delta(I,H)\right\}  .
\]
Thus, setting $y=(L^{-1})^{T}x$, one obtains%
\begin{align*}
\mathfrak{B}_{J,L}(x) &  =\lim_{\varepsilon\rightarrow0+}(-1)^{|I|}\sum
_{n\in\Delta(I,L)}\widehat{\varphi}(\varepsilon n)\frac{e^{2\pi in\cdot x}%
}{(2\pi iLn)^{J}}\\
&  =\lim_{\varepsilon\rightarrow0+}(-1)^{|I|}\sum_{n\in\Delta(I,L)}%
\widehat{\varphi}(\varepsilon L^{-1}Ln)\frac{e^{2\pi iLn\cdot y}}{(2\pi
iLn)^{J}}\\
&  =\lim_{\varepsilon\rightarrow0+}(-1)^{|I|}\sum_{m\in\Delta(I,H)}%
\widehat{\varphi}(\varepsilon L^{-1}Hm)\frac{e^{2\pi iHm\cdot y}}{(2\pi
iHm)^{J}}\\
&  =\lim_{\varepsilon\rightarrow0+}(-1)^{|I|}\sum_{m\in\Delta(I,H)}%
\widehat{\psi}(\varepsilon Hm)\frac{e^{2\pi iHm\cdot y}}{(2\pi iHm)^{J}}.
\end{align*}
Let now $H=(h_{j,k})$, $k_{j}=\prod_{s=j}^{d}h_{s,s}$ for all $j=1,\ldots,d$,
and set $K=\text{diag}(k_{1},\ldots,k_{d})$ to be the corresponding diagonal
matrix. We claim that $K\mathbb{Z}^{d}\subseteq H\mathbb{Z}^{d}$. Indeed, it
suffices to show that all vectors $Ke_{j}$ belong to $H\mathbb{Z}^{d}$.
Obviously, $Ke_{d}=h_{d,d}e_{d}=He_{d}$. By induction, assuming that
$Ke_{s}\in H\mathbb{Z}^{d}$ for all $s=j+1,j+2,\ldots,d$, let us show that
$Ke_{j}\in H\mathbb{Z}^{d}$. We have%
\begin{align*}
Ke_{j} &  =h_{j,j}\cdots h_{d,d}e_{j}\\
&  =\sum_{s=j}^{d}h_{s,j}h_{j+1,j+1}\ldots h_{d,d}e_{s}-\sum_{s=j+1}%
^{d}h_{s,j}h_{j+1,j+1}\ldots h_{d,d}e_{s}\\
&  =H(h_{j+1,j+1}\ldots h_{d,d}e_{j})-\sum_{s=j+1}^{d}h_{s,j}h_{j+1,j+1}%
...\left(  h_{s,s}\ldots h_{d,d}e_{s}\right)  \\
&  =H(h_{j+1,j+1}\ldots h_{d,d}e_{j})-\sum_{s=j+1}^{d}h_{s,j}h_{j+1,j+1}%
...h_{s-1,s-1}Ke_{s}\in H\mathbb{Z}^{d}%
\end{align*}
and the claim is proved. Observe that there is a finite number of different
integer translates of $K\mathbb{Z}^{d}$ (precisely $k_{1}k_{2}\ldots k_{d}$).
Take any point of $H\mathbb{Z}^{d}$ which is not in $K\mathbb{Z}^{d}$, say
$v^{(1)}$. By linearity it follows that $v^{(1)}+K\mathbb{Z}^{d}$ is contained
in $H\mathbb{Z}^{d}$ and is disjoint from $K\mathbb{Z}^{d}$. Take again a
second vector in $H\mathbb{Z}^{d}$ which is not in $K\mathbb{Z}^{d}%
\cup(v^{(1)}+K\mathbb{Z}^{d}),$ say $v^{(2)}$. Then $v^{(2)}+K\mathbb{Z}^{d}$
is contained in $H\mathbb{Z}^{d}$ (and is disjoint from $K\mathbb{Z}^{d}%
\cup(v^{(1)}+K\mathbb{Z}^{d})$). We can iterate this procedure until we
exhaust all of $H\mathbb{Z}^{d}$. In other words, we have%
\[
L\mathbb{Z}^{d}=H\mathbb{Z}^{d}=\bigcup\limits_{\ell=1}^{\mathcal{L}}%
(v^{(\ell)}+K\mathbb{Z}^{d})
\]
where the union is a disjoint union. Thus, recalling that $y=(L^{-1})^{T}x$,
\begin{align*}
&  \mathfrak{B}_{J,L}(x)=\lim_{\varepsilon\rightarrow0+}(-1)^{|I|}\sum
_{m\in\Delta(I,H)}\widehat{\psi}(\varepsilon Hm)\frac{e^{2\pi iHmy}}{(2\pi
iHm)^{J}}\\
&  =\lim_{\varepsilon\rightarrow0+}(-1)^{|I|}\sum_{\ell=1}^{\mathcal{L}}%
\sum_{\substack{m\in\mathbb{Z}^{d}\\v_{s}^{(\ell)}+(Km)_{s}=0\text{ iff }%
i_{s}=0}}\widehat{\psi}\left(  \varepsilon(v^{(\ell)}+Km)\right)
\frac{e^{2\pi i(v^{(\ell)}+Km)y}}{\left(  2\pi i(v^{(\ell)}+Km)\right)  ^{J}%
}\\
&  =\lim_{\varepsilon\rightarrow0+}(-1)^{|I|}\sum_{\ell=1}^{\mathcal{L}}%
\prod_{\substack{s=1\\\text{s.t. }i_{s}=1}}^{d}e^{2\pi iv_{s}^{(\ell)}y_{s}%
}\sum_{\substack{m_{s}\in\mathbb{Z}\\m_{s}\neq-\frac{v_{s}^{(\ell)}}{k_{s}}%
}}\widehat{\eta}\left(  \varepsilon(v_{s}^{(\ell)}+k_{s}m_{s})\right)
\frac{e^{2\pi ik_{s}m_{s}y_{s}}}{\left(  2\pi i(v_{s}^{(\ell)}+k_{s}%
m_{s})\right)  ^{j_{s}}}\\
&  =\sum_{\ell=1}^{\mathcal{L}}\prod_{\substack{s=1\\\text{s.t. }i_{s}=1}%
}^{d}(2\pi ik_{s})^{-j_{s}}e^{2\pi iv_{s}^{(\ell)}y_{s}}\times\\
&  \quad\times\lim_{\varepsilon\rightarrow0+}(-1)\sum_{\substack{m_{s}%
\in\mathbb{Z}\\m_{s}\neq-\frac{v_{s}^{(\ell)}}{k_{s}}}}\widehat{\eta}\left(
\varepsilon k_{s}\left(  \frac{v_{s}^{(\ell)}}{k_{s}}+m_{s}\right)  \right)
\frac{e^{2\pi im_{s}(k_{s}y_{s})}}{\left(  \frac{v_{s}^{(\ell)}}{k_{s}}%
+m_{s}\right)  ^{j_{s}}}\\
&  =\sum_{\ell=1}^{\mathcal{L}}\prod_{\substack{s=1\\\text{s.t. }i_{s}=1}%
}^{d}e^{2\pi iv_{s}^{(\ell)}y_{s}}\left(  2\pi ik_{s}\right)  ^{-j_{s}%
}L_{j_{s}}\left(  k_{s}y_{s},\frac{v_{s}^{(\ell)}}{k_{s}}\right)
\end{align*}
where for $j\geq1$ we set
\[
L_{j}(x,r)=-\lim_{\varepsilon\rightarrow0+}\sum_{n\in\mathbb{Z\setminus
}\{-r\}}\widehat{\eta}\left(  \varepsilon(n+r)\right)  \frac{e^{2\pi inx}%
}{(n+r)^{j}}.
\]
Observe that for $r\in\mathbb{Z}$
\[
L_{j}(x,r)=e^{-2\pi irx}(2\pi i)^{j}B_{j}(x)
\]
where $B_{j}(x)$ is the $j$-th Bernoulli polynomial, whereas when
$r\notin\mathbb{Z}$ the function $L_{j}(x,r)$ is related to the Lerch Zeta
function
\[
\mathfrak{L}(x,j,r)=\sum_{n=0}^{+\infty}\frac{e^{2\pi inx}}{(n+r)^{j}}.
\]
Indeed, formally,%
\[
L_{j}(x,r)=-\sum_{\substack{n\in\mathbb{Z}\\n+r\neq0}}\frac{e^{2\pi inx}%
}{(n+r)^{j}}=-\mathfrak{L}(x,j,r)-(-1)^{j}\mathfrak{L}(-x,j,-r)+r^{-j}.
\]

Moreover, when $r=p/q$ is rational, it is not difficult to write $L_{j}(x,r)$
in terms of periodic Bernoulli polynomials. Indeed, one can verify that for
every periodic integrable function $f\left(  x\right)  $,%
\[
\frac{1}{q}\sum_{a=0}^{q-1}e^{-2\pi i\frac{p}{q}\left(  x+a\right)  }f\left(
\frac{x+a}{q}\right)  =\sum_{n\in\mathbb{Z}}\widehat{f}\left(  nq+p\right)
e^{2\pi inx}.
\]
Therefore, with $f\left(  x\right)  =B_{j}\left(  x\right)  $,%
\begin{align*}
L_{j}&(x,p/q)  =-\sum_{\substack{n\in\mathbb{Z}\\n+p/q\neq0}}\frac{e^{2\pi
inx}}{(n+p/q)^{j}}=\left(  2\pi iq\right)  ^{j}\sum_{\substack{n\in
\mathbb{Z}\\nq+p\neq0}}\frac{-1}{\left(  2\pi i(nq+p)\right)  ^{j}}e^{2\pi
inx}\\
&  =\left(  2\pi iq\right)  ^{j}\sum_{n\in\mathbb{Z}}\widehat{B}_{j}\left(
nq+p\right)  e^{2\pi inx}=\left(  2\pi iq\right)  ^{j}e^{-2\pi ix\frac pq}\frac
{1}{q}\sum_{a=0}^{q-1}e^{-2\pi ia\frac pq}B_{j}\left(  \frac{x+a}{q}\right)  .
\end{align*}
Thus,%
\begin{align*}
\mathfrak{B}_{J,L}(x) &  =\sum_{\ell=1}^{\mathcal{L}}\prod
_{\substack{s=1\\\text{s.t. }i_{s}=1}}^{d}e^{2\pi iv_{s}^{(\ell)}y_{s}}\left(
2\pi ik_{s}\right)  ^{-j_{s}}L_{j_{s}}\left(  k_{s}y_{s},\frac{v_{s}^{(\ell)}%
}{k_{s}}\right)  \\
&  =\sum_{\ell=1}^{\mathcal{L}}\prod_{\substack{s=1\\\text{s.t. }i_{s}=1}%
}^{d}\frac{1}{k_{s}}\sum_{a=0}^{k_{s}-1}e^{-2\pi i\frac{v_{s}^{(\ell)}}{k_{s}%
}a}B_{j_{s}}\left(  y_{s}+\frac{a}{k_{s}}\right).
\end{align*}

Now, recalling that $K=\textrm{diag}(k_1,\ldots, k_d)$, if we set $K^I=\textrm{diag}(k_1^{i_1},\ldots, k_d^{i_d})$ so that $k_s$ is replaced with $1$ whenever $i_s=0$, then
\begin{align*}
\mathfrak{B}_{J,L}(x) &  =\sum_{\ell=1}^{\mathcal{L}}\prod_{\substack{s=1\\\text{s.t. }i_{s}=1}%
}^{d}\frac{1}{k_{s}}\sum_{a=0}^{k_{s}-1}e^{-2\pi i\frac{v_{s}^{(\ell)}}{k_{s}%
}a}B_{j_{s}}\left(  y_{s}+\frac{a}{k_{s}}\right)\\
&= \sum_{\ell=1}^{\mathcal{L}}\sum_{a_{1}=0}^{k_{1}^{i_1}-1}\frac{1}{k_{1}^{i_1}}e^{-2\pi
i\frac{v_{1}^{(\ell)}}{k_{1}^{i_1}}a_{1}}B_{j_{1}}\left(  y_{1}+\frac{a_{1}}{k_{1}^{i_1}%
}\right)  \cdots\sum_{a_{d}=0}^{k_{d}^{i_d}-1}\frac{1}{k_{d}^{i_d}}e^{-2\pi i\frac
{v_{d}^{(\ell)}}{k_{d}^{i_d}}a_{d}}B_{j_{d}}\left(  y_{d}+\frac{a_{d}}{k_{d}^{i_d}%
}\right)  \\
&=\sum_{\ell=1}^{\mathcal{L}}\det\left(  K^I\right)  ^{-1}\sum_{0\leqslant
A\leqslant K^I-1}e^{-2\pi i((K^I)^{-1})A\cdot v^{(\ell)}}\mathfrak{B}_{J,\text{Id}%
}\left(  (L^{-1})^{T}x+(K^I)^{-1}A\right),
\end{align*}
where $0\leqslant A=(a_1,\ldots,a_d)\leqslant K^I-1$ means $0\leqslant a_{s}\leqslant k_{s}^{i_s}-1$.

\bibliographystyle{plain}
\bibliography{TrasformataSd-bib}

\begin{thebibliography}{10}

\bibitem{agapito_weitsman}
J.~Agapito and J.~Weitsman.
\newblock The weighted {E}uler-{M}ac{L}aurin formula for a simple integral
  polytope.
\newblock {\em Asian J. Math.}, 9(2):199--211, 2005.

\bibitem{Aomoto}
K.~Aomoto.
\newblock Analytic structure of {S}chl\"{a}fli function.
\newblock {\em Nagoya Math. J.}, 68:1--16, 1977.

\bibitem{baldoni_berline_vergne}
V.~Baldoni, N.~Berline, and M.~Vergne.
\newblock Local {E}uler-{M}ac{L}aurin expansion of {B}arvinok valuations and
  {E}hrhart coefficients of a rational polytope.
\newblock In {\em Integer points in polyhedra---geometry, number theory,
  representation theory, algebra, optimization, statistics}, volume 452 of {\em
  Contemp. Math.}, pages 15--33. Amer. Math. Soc., Providence, RI, 2008.

\bibitem{BR}
M.~Beck and S.~Robins.
\newblock {\em Computing the continuous discretely}, volume~61.
\newblock Springer, 2007.

\bibitem{BeckRobinsSam}
M.~Beck, S.~Robins, and S.~V. Sam.
\newblock Positivity theorems for solid-angle polynomials.
\newblock {\em Beitr. Algebra Geom.}, 51(2):493--507, 2010.

\bibitem{BV07}
N.~Berline and M.~Vergne.
\newblock Local {E}uler-{M}aclaurin formula for polytopes.
\newblock {\em Mosc. Math. J.}, 7(3):355--386, 573, 2007.

\bibitem{BV}
N.~Berline and M.~Vergne.
\newblock Local asymptotic {E}uler-{M}ac{L}aurin expansion for {R}iemann sums
  over a semi-rational polyhedron.
\newblock In {\em Configuration Spaces}, pages 67--105. Springer, 2016.

\bibitem{BoasBuck}
R.~P. Boas, Jr. and R.~Creighton Buck.
\newblock {\em Polynomial expansions of analytic functions}.
\newblock Ergebnisse der Mathematik und ihrer Grenzgebiete, (N.F.), Band 19.
  Academic Press, Inc., Publishers, New York; Springer-Verlag, Berlin, 1964.
\newblock Second printing, corrected.

\bibitem{BCRT}
L.~Brandolini, L.~Colzani, S.~Robins, and G.~Travaglini.
\newblock An {E}uler-{M}ac{L}aurin formula for polygonal sums.
\newblock {\em Trans. Amer. Math. Soc.}, 375(1):151--172, 2022.

\bibitem{Brion}
M.~Brion.
\newblock Points entiers dans les poly\`edres convexes.
\newblock In {\em Annales scientifiques de l'{\'E}cole Normale Sup{\'e}rieure},
  volume~21, pages 653--663, 1988.

\bibitem{BrionVergne}
M.~Brion and M.~Vergne.
\newblock Lattice points in simple polytopes.
\newblock {\em J. Amer. Math. Soc.}, pages 371--392, 1997.

\bibitem{CS1}
S.~E. Cappell and J.~L. Shaneson.
\newblock Genera of algebraic varieties and counting of lattice points.
\newblock {\em Bull. Amer. Math. Soc. (N.S.)}, 30(1):62--69, 1994.

\bibitem{CS2}
S.~E. Cappell and J.~L. Shaneson.
\newblock Euler-{M}ac{L}aurin expansions for lattices above dimension one.
\newblock {\em C. R. Acad. Sci. Paris Sér. I Math.}, 321(7):885--890, 1995.

\bibitem{C-D-R}
L.~Colzani, L.~De~Michele, and D.~Roux.
\newblock Fourier expansions of piecewise smooth functions.
\newblock {\em Mediterr. J. Math.}, 9(2):253--266, 2012.

\bibitem{DLR}
R.~Diaz, Q.-N. Le, and S.~Robins.
\newblock Fourier transforms of polytopes, solid angle sums, and discrete
  volume.
\newblock {\em preprint arXiv:1602.08593}, 2016.

\bibitem{Edmonds}
A.~L. Edmonds.
\newblock Simplicial decompositions of convex polytopes.
\newblock {\em Pi Mu Epsilon J.}, 5:124--128, 1970.

\bibitem{Eriksson}
F.~Eriksson.
\newblock On the measure of solid angles.
\newblock {\em Math. Mag.}, 63(3):184--187, 1990.

\bibitem{Eul}
L.~Euler.
\newblock Subsidium calculi sinuum.
\newblock {\em Novi Commentarii academiae scientiarum Petropolitanae},
  5:164--204, 1752.

\bibitem{fischer_pommersheim}
B.~Fischer and J.~Pommersheim.
\newblock An algebraic construction of sum-integral interpolators.
\newblock {\em preprint arXiv:2101.04845}, 2021.

\bibitem{GP}
S.~Garoufalidis and J.~Pommersheim.
\newblock Sum-integral interpolators and the {E}uler-{M}ac{L}aurin formula for
  polytopes.
\newblock {\em Trans. Amer. Math. Soc.}, 364(6):2933--2958, 2012.

\bibitem{GS}
V.~Guillemin and S.~Sternberg.
\newblock Riemann sums over polytopes.
\newblock In {\em Annales de l'Institut Fourier}, volume~57, pages 2183--2195,
  2007.

\bibitem{KSW}
Y.~Karshon, S.~Sternberg, and J.~Weitsman.
\newblock The {E}uler-{M}aclaurin formula for simple integral polytopes.
\newblock {\em Proc. Natl. Acad. Sci. USA}, 100(2):426--433, 2003.

\bibitem{KSW2}
Y.~Karshon, S.~Sternberg, and J.~Weitsman.
\newblock {E}uler-{M}aclaurin with remainder for a simple integral polytope.
\newblock {\em Duke Math. J.}, 130(3):401--434, 2005.

\bibitem{KSW3}
Y.~Karshon, S.~Sternberg, and J.~Weitsman.
\newblock Exact {E}uler-{M}aclaurin formulas for simple lattice polytopes.
\newblock {\em Adv. in Appl. Math.}, 39(1):1--50, 2007.

\bibitem{KP}
A.~Khovanskii and A.~Pukhlikov.
\newblock Finitely additive measures of virtual polytopes.
\newblock {\em St. Petersburg Math. J.}, 4(2):337--356, 1993.

\bibitem{lFP}
Y.~Le~Floch and {\'A}.~Pelayo.
\newblock Euler-{M}ac{L}aurin formulas via differential operators.
\newblock {\em Adv. in Appl. Math.}, 73:99--124, 2016.

\bibitem{Mor2}
L.~J. Mordell.
\newblock Expansion of a function in a series of {B}ernoulli polynomials, and
  some other polynomials.
\newblock {\em J. Math. Anal. Appl.}, 15(1):132--140, 1966.

\bibitem{Mor}
L.~J. Mordell.
\newblock Expansion of a function in terms of {B}ernoulli polynomials.
\newblock {\em J. Lond. Math. Soc.}, 1(1):526--528, 1966.

\bibitem{Ribando}
J.~M. Ribando.
\newblock Measuring solid angles beyond dimension three.
\newblock {\em Discrete Comput. Geom.}, 36(3):479--487, 2006.

\bibitem{Sinai}
S.~Robins.
\newblock A friendly invitation to {F}ourier analysis on polytopes.
\newblock {\em preprint arXiv:2104.06407}, 2021.

\bibitem{Rudin}
W.~Rudin.
\newblock {\em Fourier analysis on groups}.
\newblock Dover Publications, 2017.

\bibitem{RTR}
J.~D. Sally and P.~J. Sally, Jr.
\newblock {\em Roots to research}.
\newblock American Mathematical Society, Providence, RI, 2007.
\newblock A vertical development of mathematical problems.

\bibitem{T}
T.~Tate.
\newblock Asymptotic {E}uler-{M}ac{l}aurin formula over lattice polytopes.
\newblock {\em J. Funct. Anal.}, 260(2):501--540, 2011.

\end{thebibliography}

\end{document}